\documentclass[reqno,12pt]{amsart}
\usepackage{amssymb,amsthm,amsmath,amsfonts,amstext,mathrsfs,fancybox}
\usepackage{latexsym}
\usepackage{epsfig}

\usepackage{fancyhdr}
\usepackage[unicode]{hyperref}

\hypersetup
{
colorlinks=true,
linkcolor=blue,
citecolor=blue,
urlcolor=blue,
filecolor=blue,
bookmarksnumbered=true,
pdfstartview=FitH,
pdfhighlight=/N}

\newtheorem{cor}{Corollary}
\newtheorem{lem}{Lemma}
\newtheorem{prop}{Proposition}

\newtheorem*{thmm}{Theorem}

\input xy
\xyoption{all}

\newcommand{\m}{\mathbf}

\newcommand{\bl}{\bullet}
\newcommand{\p}{\partial}
\def\d{\,{\rm{d}}}

\title[Rational projective plane flows]
{The projective translation equation\\ and rational plane flows. I}
\author[G. Alkauskas]{Giedrius Alkauskas}
\address{Vilnius University, Department of Mathematics and Informatics, Naugarduko 24, LT-03225 Vilnius, Lithuania}
\email{giedrius.alkauskas@gmail.com}

\newcounter{noteno}

\newenvironment{Note}%
	{\refstepcounter{noteno}%
	\begin{small}
	\medbreak\par\noindent{{\bf Note~\thenoteno}.}}%
	{\hfill{$\Box$}\end{small}\par\medbreak}

\begin{document}
\begin{abstract}Let $\mathbf{x}=(x,y)$. A {\it plane  flow} is a function $F(\mathbf{x},t):\mathbb{R}^2\times\mathbb{R}\mapsto\mathbb{R}^{2}$ such that $F(F(\mathbf{x},s),t)=F(\mathbf{x},s+t)$ for (almost) all real numbers $x,y,s,t$ (the function $F$ might not be well-defined for certain $\mathbf{x},t$). In this paper we investigate rational plane flows which are of the form $F(\m{x},t)=\phi(\m{x}t)/t$; here $\phi$ is a pair of rational functions in $2$ real variables. These may be called \emph{projective flows}, and for a description of such flows only the knowledge of Cremona group in dimension 1 is needed. Thus, the aim of this work is to completely describe over $\mathbb{R}$ all rational solutions of the two dimensional translation equation $(1-z)\phi(\m{x})=\phi(\phi(\m{x}z)(1-z)/z)$. We show that, up to conjugation with a $1-$homogenic birational plane transformation ($1-$BIR), all solutions are as follows: a zero flow, two singular flows, an identity flow, and one non-singular flow for each non-negative integer $N$, called \emph{the level} of the flow. The case $N=0$ stands apart, while the case $N=1$ has special features as well. Conjugation of these canonical solutions with $1$-BIR produce a variety of flows with different properties and invariants, depending on the level and on the conjugation itself. We explore many more features of these flows; for example, there are $1$, $4$, and $2$ essentially different symmetric flows in cases $N=0$, $N=1$, and $N\geq 2$, respectively. Many more questions related to rational flows will be treated in the second part of this work.
\end{abstract}

\date{January 04, 2012. {\it Last update\/}: November 22, 2012.}
\thanks{The author gratefully acknowledges support by the Lithuanian Science Council whose postdoctoral fellowship
 is being funded by European Union Structural Funds project ``Postdoctoral Fellowship Implementation in Lithuania".}

\maketitle

\setcounter{tocdepth}{2} 
\tableofcontents
\section*{Terminology}Many of the classical terms from geometry are new in the field of translation equation, so we briefly remind the definitions. Also, we introduce some new terms.\\

\noindent
\underline{$1-$BIR}: A birational transformation of affine space such that it is given by a collection of $1$-homogenic ratioal functions.\\
\underline{Birational transformation}: A map from the affine or projective space to itself which is rational, and whose inverse is also rational.\\
\underline{Affine translation equation}: the functional equation (\ref{gen}).\\
\underline{Affine flow}:  a map, a function, or a formal expression which satisfies the functional equation (\ref{gen}).\\
\underline{Cremona group}: A group of birational transformations of the projective space $P^{n}(\mathbb{C})$.\\
\underline{Involution}: A rational map of order $2$; genarally, an element of a group or order $2$.\\
\underline{Jonqui\`{e}res transformation}: a function in $2$ variables of the form (\ref{jonq}).\\
\underline{Flow in univariate form}: a projective flow of the form $\big{(}\mathcal{U}(x,y),\frac{y}{y+1}\big{)}$; this (with a slight abuse of language), refers only to the first coordinate $\mathcal{U}(x,y)$.\\
\underline{Level of a projective flow}: orbits of a rational projective flow are given by the equations of the form 
$\mathscr{W}(x,y)=c$, where $\mathscr{W}$ is a homogenic rational function; the level is the least possible and non-negative value of the degree of homogenuity (since $\mathcal{W}^{r}(x,y)=c^r$ are also equations for the orbits for $r\in\mathbb{Z}\setminus\{0\}$). The rational vector field can give rise to non-rational flows whose orbits are still rational. In this case, the level also refers to the degree of $\mathscr{W}$.\\
\underline{$\ell$-conjugacy of flows}: two flows of the form $\phi(\m{x})$ and $\ell^{-1}\circ\phi\circ\ell(\m{x})$, where $\ell$ is a $1$-BIR.\\
\underline{$\l$-conjugacy of flows}: the same as $\ell-$conjugacy, only this time $\ell$ must be a non-degenerate linear transformation.\\
\underline{M\"{o}bius transformations}: transformations of $P^{1}(\mathbb{R})$ or $P^{1}(\mathbb{C})$ given by $x\mapsto\frac{ax+b}{cx+d}$, $ad-bc\neq 0$.\\
\underline{Orbit of the flow}: For affine flows, this is the set $F(\m{x},t)$ for fixed $\m{x}$, and $t\in\mathbb{R}$ or $\mathbb{C}$. For projective flows, this is the set $\frac{1}{t}\phi(\m{x}t)$, also for fixed $\m{x}$ and $t\in\mathbb{R}$ or $\mathbb{C}$. \\
\underline{Projective translation equation, PrTE}: the functional equation (\ref{funk}).\\
\underline{Projective flow}: a map, a function, or a formal expression which satisfies the functional equation (\ref{funk}).\\
\underline{Representation set of a flow}: a set which contains a single point from each orbit; this applies only to a topological case.\\
\underline{Vector field of a projective flow}: a vector field whose vectors are tangent to the orbits with the lenght indicating the velocity.\\
\underline{Notation}. $\phi=u\bl v$ - a projective flow. $\phi_{N}$ - the canonical flow of level $N$. $\varpi$, $\varrho$ - the first and the second coordinates of the vector field; these are $2$-homogenic rational functions. $\mathscr{F}_{N}$ - the set of projective flows in univariate form of level $N$. ${\tt FL}_{N}$ - the set of projective flows of level $N$. $\mathscr{W}$ - equation for an orbit of a flow. $\mathscr{V}$ - an orbit of a flow. 
\section{Introduction and main results}
To simplify our notation, if $\mathbf{F}:\mathbb{R}^{2}\mapsto\mathbb{R}^{2}$ is any $2-$dimensional function, we write $\mathbf{F}(x,y)=G(x,y)\bl H(x,y)$ instead of $\mathbf{F}(x,y)=\Big{(}G(x,y),H(x,y)\Big{)}$. This paper is a sequel to \cite{alkauskas}, though it is independent and can be read separately. The second part of the current work is in preparation \cite{alkauskas2}; meanwhile, a direct continuation of the current work is \cite{alkauskas-un}. There we explore flows that are not rational but still have a rational vector field, and we are essentially using the results of the Section \ref{secfin}. So, in several places of the current study we include short notes relevant to the questions discussed; these remarks emerged after \cite{alkauskas-un} was finished.

\subsection{Introduction}

Let $\m{x}=(x,y)$ be a point in $\mathbb{R}^2$. Assume that a $2-$dimensional function $F(\m{x},t)$ describes a motion of a point in the plane: it gives the coordinates of the point at the time moment $t$, provided that at the time moment $t=0$ the point is at the position $\m{x}$: $F(\m{x},0)=\m{x}$. We make an assumption that the time is additive: an elapse of time $s$ and a consequent immediate elapse of $t$ is the same as a total elapse of $s+t$. It is nevertheless surprising that such a seemingly scarce information forces $F$ to satisfy a natural iterative functional equation. Indeed, assume that the motion is well defined. Consider the time moment $s+t$. By the definition, the coordinates of the point are given by $F(\m{x},s+t)$. On the other hand, at the moment $s$ the coordinates are $F(\m{x},s)$. After an additional time $t$, these coordinates are equal to $F(F(\m{x},s),t)$. Thus,
\begin{eqnarray}
F(\m{x},s+t)=F(F(\m{x},s),t),\m{x}\in\mathbb{R}^{n},s,t\in\mathbb{R}.
\label{gen}
\end{eqnarray}
(In our case $n=2$). Such a function $F$ is also called a \emph{flow}. The equation (\ref{gen}) is a well-known translation equation. For its basic properties, we refer to \cite{aczel} and the two expository articles \cite{moszner1,moszner2} which contain long bibliographical lists.  The book \cite{nikolaev} deals with flows from the point of view of dynamical systems.\\

In this paper we take an algebro-geometric point of view and ask for the solutions of the translation equation which are rational functions in $x$, $y$ and $t$. Let $F(\m{x},1)=\phi(\m{x})$, $F(\m{x},-1)=\psi(\m{x})$. When $s=1$, $t=-1$, the translation equation (\ref{gen}), if we assume $F(\m{x},0)=\m{x}$,  gives $\m{x}=\psi\circ\phi(\m{x})$, and $s=-1$, $t=1$ gives $\m{x}=\phi\circ\psi(\m{x})$. Thus, both $\psi$ and $\phi$ are mutually inverse birational space transformations. We first note the following. Let $\ell$ be a birational transformation of $\mathbb{R}^{n}$. Then, if $F(\m{x},t)$ is a flow, then so is $F^{\ell}(\m{x},t):=\ell^{-1}\circ F(\ell(\m{x}),t)$. The validity of this claim is verified by a direct check. We say that $F^{\ell}$ is obtained from $F$ by {\it a conjugation} with $\ell$.  One class of rational plane flows can be given as follows. Let $f$ be a birational transformation of $\mathbb{R}^{n}$, and $\m{a}\in\mathbb{R}^{n}$ is any given vector. Then
\begin{eqnarray}
F(\m{x},t)=f(f^{-1}(\m{x})+\m{a}t).
\label{normal}
\end{eqnarray}
satisfies the $n-$dimensional translation equation (\ref{gen}) with the initial condition $F(\m{x},0)=\m{x}$. Note that $f$, as a function, can be non-invertible (for example, pre-images of certain points might contain more than one point) and not well-defined on the whole space $\mathbb{R}^n$. Here ``birational" means that $f$ possesses a formal rational inverse $g$: $g\circ f(\m{x})=f\circ g(\m{x})=\m{x}$. We do not know whether these examples cover all rational flows. For a fixed $\m{x}_{0}$, consider the set $F(\m{x}_{0},t)$, $t\in\mathbb{R}$, and call it \emph{the orbit} of the point $\m{x}_{0}\in\mathbb{R}^{n}$. Suppose, there exists a hyperplane of dimension $n-1$ which contains a single representative from each orbit. After a conjugation with an invertible affine map we may assume that this is the hyperplane $x_{1}=0$, where $\m{x}=(x_{1},\ldots,x_{n})$. Let $\m{\xi}\in\mathbb{R}^{n-1}$. Thus, according to our assumption, for any given $\m{y}\in\mathbb{R}^{n}$, the equation
\begin{eqnarray*}
F((0,\m{\xi}),s)=\m{y}
\end{eqnarray*}
is uniquely solvable in $\xi$, $s$. Therefore, let $f(\m{y})=(s,\m{\xi})$ be the inverse map. If this is a rational function, then indeed, the results in (\cite{aczel}, p. 367, see also \cite{berg}) imply that the flow $F$ is given in the form (\ref{normal}). Despite this, we already know that a construction of a hyperplane which contains a single point from each orbit (such sets are called \emph{representation sets}) is a difficult task, and (in another setting) we managed to do this only in dimension $n=2$ (\cite{alkauskas}, p. 346). The dimension $n=1$, though, can be fully solved (\cite{aczel}, p. 245). For $n=1$, all birational line transformations are given by the M\"{o}bius fractional functions, and thus all rational $1-$dimensional flows, up to conjugation with $\ell(x)=cx+d$, $c\neq 0$, are
\begin{eqnarray*}
\text{either }F(x,t)&=&x+t, \text{ in the form (\ref{normal}) is is obtained from }\m{a}=1,f(x)=x;\\
\text{or }F(x,t)&=&\frac{x}{tx+1},\text{ in the form (\ref{normal}) is is obtained from  }\m{a}=1,f(x)=\frac{1}{x}.
\end{eqnarray*}
These can be thought as flows defined on a circle $P^{1}(\mathbb{R})$.
\subsection{The projective translation equation}The dimension $2$ is much more interesting and complicated. For the complete description of all flows one needs to know, among other things, the structure of the group of birational plane transformations. As a matter of fact, there is an intermediate case, where such knowledge is not needed. Namely, we will describe all rational $2-$dimensional flows $F(\m{x},t)$ where a function $F$ is of the form $F(\m{x},t)=\frac{1}{t}\phi(\m{x}t)$. As appears, this choice is fundamental in projective geometry as is suggested by the following principle: the general flows $F(\m{x},t)$ in dimension $n$ correspond to birational transformations of the affine space $\mathbb{R}^{n}$, while the flows of the form $\frac{1}{t}\phi(\m{x}t)$ correspond to birational transformations of the projective space $P^{n-1}(\mathbb{R})$, although in a different and much more complicated way. Thus, the purpose of this paper is to completely describe such rational functions $\phi(x,y)=u(x,y)\bl v(x,y)$, which satisfy

\begin{eqnarray}\setlength{\shadowsize}{2pt}\shadowbox{$\displaystyle{\quad
(1-z)\phi(\m{x})=\phi\Big{(}\phi(\m{x}z)\frac{1-z}{z}\Big{)}}$.\quad}\label{funk}
\end{eqnarray}

 Henceforth the word ``flow" will always mean ``projective flow". In this paper the word ``function" refers mostly to a rational abstract function and not a map, since rational function can be not well defined on certain curves. In few places, however, we will need to explore definition and value domains of these functions as well. We further note that there are interesting projective flows defined by algebraic, entire, or meromorphic functions; for example,
\begin{eqnarray*}
\phi_{e}(\m{x})=xe^{y}\bl y,\quad \phi_{t}(\m{x})=\frac{xy+y^2\tan y}{y-x\tan y}\bl y,\quad \phi'_{e}(\m{x})=\frac{xye^{y}}{x+y-xe^y}\bl y.
\end{eqnarray*}
Verification that these are flows are equivalent to addition formulas for $\exp(x)$ and $\tan(x)$ functions. The broadening of the meaning of (\ref{funk}) might lead to other unexpected areas of mathematics - superfunctions, iteration theory, number theory, and so on. 
In this paper, however, we deal only with rational solutions, though most often, while integrating a rational vector field, we arrive at non-rational solution.
\begin{Note}\label{note1} It appears that the class of these three flows $\phi_{e}$, $\phi_{t}$ and $\phi'_{e}$ will constitute one item in the $6$ item classification list of all unramified projective flows with rational vector fields; see further for the definition of a vector field of a flow. There are also two flows expressable in terms of elliptic functions with hexagonal lattice and one flow expressable in terms of elliptic function with square lattice as $3$ seperate items in this list; see the Note \ref{note5} in the Subsection \ref{sub4.3} for more details. It turns out that in general, the functional equation (\ref{funk}) provides a uniform framework for addition formulas for the functions $\exp(z)$, $\tan(z)$, these special ellipitc functions, and also the $\log(z)$ function (see the Subsection \ref{sub5.2}; this flow however is ramified). All this concerns just the dimension $2$, higher dimensions being open; for example, in dimension 3 the functional equation (\ref{funk}) can encode addition formulas for all elliptic functions with arbitrary period lattice, not just square and hexagonal as in the dimension 2. \end{Note}
\subsection{Singular and non-singular solutions}\label{sub1.3}The substitution $z=1$ into (\ref{funk}) gives
\begin{eqnarray*}
\phi(\m{0})=\m{0}.
\end{eqnarray*}
Let $\phi$ be a rational function and suppose it satisfies (\ref{funk}). Substitute $\m{x}/z$ instead of $\m{x}$ in (\ref{funk}). When $z\rightarrow \infty$, the right hand side tends to $\phi(-\phi(\m{x}))$, so the left hand side must tend to the same limit as well. Thus, there exists
\begin{eqnarray}
\lim\limits_{z\rightarrow 0}\frac{\phi(\m{x}z)}{z}:=r(\m{x})=-\phi(-\phi(\m{x})).
\label{ribin}
\end{eqnarray}
Further,
\begin{eqnarray*}
r(a\m{x})=\lim\limits_{z\rightarrow 0}\frac{\phi(\m{x}az)}{z}=a\lim\limits_{z\rightarrow 0}\frac{\phi(\m{x}az)}{az}=ar(\m{x}).
\end{eqnarray*}
Thus, $r(\m{x})$ is a pair of $1-$homogenic functions. In (\ref{funk}) itself, take the limit $z\rightarrow 0$. This gives
\begin{eqnarray}
\phi(\m{x})=\phi(r(\m{x})).\label{lygtis}
\end{eqnarray}
Next, let us substitute $r(\m{x})$ instead of $\m{x}$ in (\ref{ribin}). This, together with (\ref{lygtis}), gives
\begin{eqnarray}
r(r(\m{x}))=r(\m{x}).\label{iter}
\end{eqnarray}
Since $r(\m{x})$ is $1-$homogenic, the closure of the set $r(\mathbb{R}^{2})$ is invariant under homothety, and thus either consists only of the origin, it is either a line through the origin, or the whole plane. In the latter case (\ref{iter}) gives $r(\m{x})=\m{x}$. Therefore, our main assumption of this paper is that
\begin{eqnarray}
\lim\limits_{z\rightarrow 0}\frac{\phi(\m{x}z)}{z}=\m{x}.
\label{init}
\end{eqnarray}
This is exactly a non-singular case. In the beginning of Section \ref{sec2} we will show that the remaining two degenerate cases can be fully described. 
\subsection{Main results}In our investigations, the crucial property of flows is the following
\begin{prop}
 Let $\phi$ be a flow, and let $\ell$ be a 1-homogenic birational plane transformation (in short, 1-BIR). Then $\phi^{\ell}:=\ell^{-1}\circ\phi\circ \ell$ is a solution to (\ref{funk}) as thus a flow.
\label{prop1}
\end{prop}
It is proved by an easy direct check exactly the same way as it was done for homothetic functions in \cite{alkauskas}. The structure of 1-BIR is given in the Appendix \ref{app}. Let $P(x,y)$ and $Q(x,y)$ be two homogenic polynomials of degree $d\geq 0$. Then
\begin{eqnarray}
\quad\ell_{P,Q}(x,y)=\frac{xP(x,y)}{Q(x,y)}\bl\frac{yP(x,y)}{Q(x,y)}
\label{bir}
\end{eqnarray}
is a 1-BIR. Moreover, any 1-BIR can be given by $\ell_{P,Q}\circ L$, where $L$ is an invertible linear change. In the setting of Proposition \ref{prop1}, we say that then the two flows $\phi$ and $\phi^{\ell}$ are $\ell-$conjugate. If $\ell$ is a non-degenerate linear change, we will refer to this as $l-$conjugacy.\\

The main result of this paper is the following
\begin{thmm}
Let $\phi(x,y)=u(x,y)\bl v(x,y)$ be a pair of rational functions in $\mathbb{R}(x)$ which satisfies the functional equation (\ref{funk}) and the boundary condition (\ref{init}). Assume that $\phi(\m{x})\neq \phi_{{\rm id}}(\m{x}):=x\bl y$.
Then there exists an integer $N\geq 0$, which  is the basic invariant of the flow, called \emph{the level}. Such a flow $\phi(x,y)$ can be given by
\begin{eqnarray*}
\phi(x,y)=\ell^{-1}\circ\phi_{N}\circ\ell(x,y),
\end{eqnarray*}
where $\ell$ is a $1-$BIR, and $\phi_{N}$ is the canonical solution of level $N$ given by
\begin{eqnarray}
\phi_{N}(x,y)=x(y+1)^{N-1}\bl \frac{y}{y+1}.\label{can}
\end{eqnarray}
\label{mthm}
\end{thmm}
We note that since the identity flow $\phi_{{\rm id}}(\m{x})=x\bl y$ does not change under conjugation with any $1-$BIR, but it is a rational flow, it will always constitute a conjugation class of rational flows on its own; we henceforth exclude this case from our analysis without mentioning further. Suppose, $\phi$ satisfies (\ref{funk}). For $\m{x}\in\mathbb{R}^{2}$, the set
\begin{eqnarray*}
\mathscr{V}(\m{x})
=\Big{\{}\frac{1}{z}\cdot\phi(\m{x}z):z\in\mathbb{R}\Big{\}},
\end{eqnarray*}
is called \emph{the orbit} of $\m{x}$. Our Theorem implies the following result about the orbits.
\begin{cor}
If $\phi$ is a level $N$ flow, then there exists a homogenic $N$th degree rational function $\mathscr{W}(u,v)$ such that the orbit of any given $\m{x}=x\bl y$ is an algebraic genus $0$ curve $\mathscr{W}(u,v)\equiv\mathscr{W}(x,y)\equiv{\rm const.}\in\mathbb{R}\cup\{\infty\}$.
Conversely: in order for a homogeneous rational function $\mathscr{W}$ of degree $N$ to serve as an equation $\mathscr{W}(u,v)\equiv{\rm const.}$ for the orbit of a point of a certain flow, it is necessary and sufficient that
\begin{eqnarray*}
\mathscr{W}(u,v)=\frac{P^{N}(u,v)}{Q^{N}(u,v)}\cdot uv^{N-1}\circ L(u,v);
\end{eqnarray*}
here $L$ is a non-degenerate linear transformation of a plane, and $P$ and $Q$ are homogeneous polynomials of the same degree $d\geq 0$.
\label{corr1}
\end{cor}
Now we introduce our main object associated with a flow $\phi$. Let us define the associated vector field by
\begin{eqnarray*}
v(\phi;\m{x})=\varpi(x,y)\bl \varrho(x,y)=\lim\limits_{z\rightarrow 0}\frac{z^{-1}\phi(\m{x}z)-\m{x}}{z}.
\end{eqnarray*}
\begin{prop}
The vector field $v(\phi;\m{x})$ of a flow $\phi$ is a pair of $2$-homogenic rational functions. Conversely: in order for a pair of $2-$homogenic rational functions $\varpi(x,y)\bl\varrho(x,y)$ to arise from a rational flow the necessary and sufficient condition is that either $y\varpi(x,y)-x\varrho(x,y)\equiv 0$, which gives level $0$ flow, $y\varpi_{\star}(x,y)-x\varrho_{\star}(x,y)\equiv 0$, which gives level $1$ flow (here ``$\star$" means $x$ or $y$), or
\begin{eqnarray*}
\frac{y\varpi_{y}(x,y)-x\varrho_{y}(x,y)}{y\varpi_{x}(x,y)-x\varrho_{x}(x,y)}=\frac{ax+by}{cx+dy}
\end{eqnarray*}
for certain $a,b,c,d\in\mathbb{R}$ satisfying $a\neq d$, and
\begin{eqnarray*}
\frac{(a+d)^2-4bc}{(a-d)^2}=N^2,\text{ for a certain }N\in\mathbb{N}.
\end{eqnarray*}
In this case, the integer $N$ is exactly the level of a flow.
\label{ppp}
\end{prop}

Next, the following (at first sight unexpected) proposition shows that flows can be described in terms of system of PDEs.
\begin{prop}\label{propPDE}
If a rational function $\phi(x,y)=u(x,y)\bl v(x,y)$ satisfies (\ref{funk}) and (\ref{init}), then the pair $u,v$ automatically satisfies the system of second order partial differential equations, given by
\begin{eqnarray*}
\left\{\begin{array}{c@{\qquad}l}
(xu_{xx}+yu_{xy})(vu_{y}-uv_{y})+
(yu_{yy}+xu_{xy})(uv_{x}-vu_{x})=
2(u-xu_{x}-yu_{y})(u_{x}v_{y}-u_{y}v_{x}),\\
(xv_{xx}+yv_{xy})(vu_{y}-uv_{y})+
(yv_{yy}+xv_{xy})(uv_{x}-vu_{x})=
2(v-xv_{x}-yv_{y})(u_{x}v_{y}-u_{y}v_{x}).
\end{array}\right.
\end{eqnarray*}
Conversely: any pair of rational functions $u,v$, satisfying the above system and the boundary condition (\ref{init}), defines a flow $\phi(\m{x})=u(x,y)\bl v(x,y)$.
\end{prop}
Because of some natural constrains, our proof of this proposition is valid only for rational functions: the current work is really about birational geometry rather than the theory of PDEs.
 \begin{cor} The rational solution to the system in Proposition \ref{propPDE} with the boundary condition (\ref{init}) is automatically a birational plane transformation.
 \label{corr2}
\end{cor}
We further note that minding the Proposition \ref{ppp} and the identities (\ref{vfield}), the level can be calculated directly if the flow is given.
\subsection{Flows in univariate form}\label{subuni}We will also demonstrate the validity of the following
\begin{prop}
For any rational flow $\phi(\m{x})$ there exists a $1-$BIR $\ell$ of the form
\begin{eqnarray}
\ell(x,y)=xA(x,y)\bl y A(x,y),
\label{birat}
\end{eqnarray}
 where $A$ is a $0-$homogenic rational function, such that
\begin{eqnarray}
\ell^{-1}\circ\phi\circ\ell(x,y)=\mathcal{U}(x,y)\bl\frac{y}{y+1}.
\label{univ}
\end{eqnarray}
\label{propuni}
\end{prop}
We call $\mathcal{U}(x,y)$ \emph{the univariate form} of the flow $\phi$, and $\mathcal{U}(x,y)$ itself \emph{an univariate flow}. Note that the second coordinate of this flow is $y(y+1)^{-1}$, which is a birational transformation of $P^{1}(\mathbb{R})$. Therefore, such $\mathcal{U}(x,y)$ is necessarily a Jonqui\`{e}res transformation and is of the form
\begin{eqnarray}
\label{jonq}
\mathcal{U}(x,y)=\frac{a(y)x+b(y)}{c(y)x+d(y)},\quad a,b,c,d\in\mathbb{R}(y),\quad a(y)d(y)-b(y)c(y)\text{ is not identically }0.
\end{eqnarray}
The difference with the Theorem is that now we require a $1-$BIR $\ell(x,y)$ to be of the special form (\ref{birat}) rather than arbitrary. These, as already mentioned, are closely related and in fact invertible linear maps followed by (\ref{birat}) give all 1-BIRs (see Appendix \ref{app}). Let us introduce
\begin{eqnarray*}
\mathscr{F}_{N}=\Big{\{}\mathcal{U}(x,y)\in\mathbb{R}(x,y):\mathcal{U}(x,y)\bl\frac{y}{y+1}\text{ is a flow of level }N\Big{\}}\subset\mathbb{R}(x,y).
\end{eqnarray*}
In most cases, the function $\ell$ described in Proposition \ref{propuni} is unique. Nevertheless, there is one exception.
Suppose for a given flow, there exist at least two non-proportional $1-$BIRs of the form (\ref{birat}) such that (\ref{univ}) holds for two different rational functions $\mathcal{U}(x,y)$ respectively. We will later show that in this case the set of such $\ell$ is infinite, and they are given by a one parameter family
\begin{eqnarray*}
\ell_{\sigma}(x,y)=\ell_{1}(x,y)+\sigma\cdot\ell_{0}(x,y),\quad \sigma\in\mathbb{R}.
\end{eqnarray*}
In fact, this holds exactly if $\phi$ is a level $1$ flow.
Thus, if $\phi$ such a flow, then
\begin{eqnarray}
\ell_{\sigma}^{-1}\circ\phi\circ\ell_{\sigma}(x,y)=\mathcal{U}_{\sigma}(x,y)\bl \frac{y}{y+1},\quad\sigma\in\mathbb{R},\quad \mathcal{U}_{\sigma}(x,y)\in\mathbb{R}(x,y).
\label{super}
\end{eqnarray}

Let ${\tt FL}_{N}$ be the set of all level $N$ flows. Let $\phi^{z}(\m{x})=\frac{1}{z}\phi(z\m{x})$, $z\in\mathbb{R}$ is fixed. This coincides with the notation of Proposition \ref{prop1} for the linear map $\ell(x,y)=(zx,zy)$.
\begin{prop}\label{mthm2}
The set $\mathscr{F}_{0}$ is a singleton, and
\begin{eqnarray*}
\frac{x}{y+1}\in\mathscr{F}_{0}.
\end{eqnarray*}

Further, there exists the canonical map from the space $\mathscr{F}_{N}$  to the one sheeted hyperboloid
\begin{eqnarray}
\mathcal{H}_{N}=\{(X,Y,Z)\in\mathbb{R}^{3}:4XZ=(Y+1)^2-N^2\},
\label{hyper}
\end{eqnarray}
and this map is a homeomorphism.\\

For $N\geq 2$, there exists the canonical map
\begin{eqnarray*}
p_{N}:{\tt FL}_{N}\mapsto\mathcal{H}_{N}
\end{eqnarray*}
with the property that $p_{N}(\phi)=p_{N}(\psi)$ if only if $\phi$ and $\psi$ are $\ell-$conjugate for $\ell$ of the form (\ref{birat}). In particular, $p_{N}(\phi^{z})=p_{N}(\phi)$ for any real $z$.\\

For $N=1$ such map is not available, but one can construct the canonical map $\hat{p}:{\tt FL}_{1}\mapsto P^{1}(\mathbb{R})$ with a property $\hat{p}(\phi^{z})=\hat{p}(\phi)$ for any real $z$ by the following rule: for a flow of level $1$ there exists the unique $\ell$ of the form (\ref{birat}) such that
\begin{eqnarray*}
\ell^{-1}\circ\phi\circ\ell\text{ is equal either to }\frac{x}{\tau x+1}\bl\frac{y}{y+1},\text{ or }\frac{x}{(y+1)^2}\bl\frac{y}{y+1}
\end{eqnarray*}
for a certain $\tau\in\mathbb{R}$. We set $\hat{p}(\phi)=\tau$ in the first case and $\hat{p}(\phi)=\infty$ in the second one.
\end{prop}
These maps $p_{N}$ and $\hat{p}$ will turn out to be continuous when we introduce the topology to the spaces ${\tt FL}_{N}$, or, in other words, the notion of ``closeness" of two flows of the same level.\\

 This paper is organized as follows. In the Subsection \ref{sub2.1} we deal with the degenerate case in the setting of the Subsection \ref{sub1.3}. In the rest of the Section \ref{sec2} we present many examples of flows of various levels. The Section \ref{sec3} introduces the main tools for dealing with flows - vector fields and PDEs. In particular, the Proposition \ref{propPDE} is proved in the Subsection \ref{sub3.2}. The proof of main Theorem occupies all the Section \ref{secfin}, except for the last two subsections. In the Subsection \ref{sub4.5} the Proposition \ref{mthm2} is proved, and in the Subsection \ref{sub4.6} we prove the Proposition \ref{ppp}. The last chapter and the appendix contain many complementary results; in particular, we find all rational flows $\phi(\m{x})=u(x,y)\bl v(x,y)$ which are
symmetric : $v(x,y)=u(y,x)$. This study will be followed by the second part, and main directions of investigations are summarized in the end of the Section \ref{sec5}.
\begin{Note} It turns out that the current paper is about the genus $0$ case. Rational projective flows can be generalized purely algebraically without an appeal to abelian functions (elliptic or trigonometric functions in dimension $2$) to include flows whose orbits are plane curves of the form $\mathscr{W}(x,y)=c$ for a degree $N$ homogenic function $\mathscr{W}(x,y)$, but now these curves may have an arbitrary genus. These generalizations of flows turn out to be pairs $(U,V)$ of $1-$homogeneous rational functions in $4$ variables which satisfy the boundary conditions, and also each satisfies the linear PDE, but  only modulo a certain $4$ variable polynomial rather than identically. These $4-$variable rational functions can be called \emph{$2-$dimensional quasi-rational projective flows}. This question is briefly discussed in \cite{alkauskas-un}, and this is just the $2-$dimensional part of a wider theory.
\end{Note}
\section{Basic examples and properties of flows}\label{sec2}
\subsection{The degenerate case}\label{sub2.1}First, we prove the following
\begin{prop}
Any rational solution of (\ref{funk}) which does not satisfy (\ref{init}) is either $\phi(\m{x})=\m{0}$, or it is of the form
\begin{eqnarray*}
\phi(\m{x})=A\frac{R(x,y)}{cR(x,y)+1}\bl B\frac{R(x,y)}{cR(x,y)+1},
\end{eqnarray*}
where $R(x,y)$ is a $1-$homogeneous rational function, and $A,B,c$ are any real numbers satisfying $R(A,B)=1$.
\end{prop}
\begin{proof}As noted in the Subsection \ref{sub1.3}, we are left with two cases: if $r(\m{x})$ is defined by (\ref{ribin}), and $\phi$ does not satisfy (\ref{init}), then either $r(\m{x})\equiv \m{0}$, or the closure of $r(\mathbb{R}^{2})$ is the line through the origin. In the first case the identity (\ref{lygtis}) implies $\phi(\m{x})=\phi(\m{0})=\m{0}$. In the second case, let $\phi(\m{x})=u(x,y)\bl v(x,y)$, and let the line $r(\mathbb{R}^{2})$ be $cx+dy=0$. If $L$ is a non-degenerate linear transformation, the $L^{-1}\circ\phi\circ L$ also satisfies (\ref{funk}). Thus, up to conjugation by a linear map, we may assume that the closure of the set $r(\mathbb{R}^{2})$ is the line $x=0$.
In this case, (\ref{lygtis}) shows that $\phi(\mathbb{R}^{2})=\phi(\{0\}\times\mathbb{R})$ (we always take the closures of the sets under consideration), and thus this set is an algebraic curve $\{u(0,t)\bl v(0,t):t\in\mathbb{R}\}$. However, the equation (\ref{funk}) shows that the image $\phi(\mathbb{R}^2)$ is invariant under homothety, and so this curve must be a line through the origin. So, $\phi(\m{x})=Av(x,y)\bl Bv(x,y)$ for a certain rational function $v(x,y)$. Since $\lim_{z\rightarrow 0}\phi(\m{x}z)/z=r(\m{x})=0\bl f(x,y)$ for a certain rational $f(x,y)$, this shows that $A=0$, and without loss of generality, we may assume $B=1$. Consider the function $v(0,y)$. Substitution of $x=0$ into (\ref{funk}) shows that $v(0,y)$ is a $1-$dimensional rational flow, and thus we know that
\begin{eqnarray}
v(0,y)=\frac{y}{cy+1}\text{ for a certain }c\in\mathbb{R}.\label{secco}
\end{eqnarray}
The equation (\ref{funk}), taking care only of the second coordinate, reads as
\begin{eqnarray*}
(1-z)v(x,y)=v\Big{(}0,\frac{1-z}{z}\cdot v(xz,yz)\Big{)}.
\end{eqnarray*}
Using (\ref{secco}), this gives
\begin{eqnarray*}
v(x,y)=\frac{v(xz,yz)}{c(1-z)v(xz,yz)+z}.
\end{eqnarray*}
This identity claims exactly that the function $1/v(x,y)-c$ is $(-1)-$homogeneous. So, there exists a $1-$homogeneous rational function $R(x,y)$, such that
\begin{eqnarray*}
v(x,y)=\frac{R(x,y)}{cR(x,y)+1}.
\end{eqnarray*}
The condition (\ref{secco}) requires that $R(0,y)=y$, or, which amounts to the same, $R(0,1)=1$.
A direct calculation shows that if $u(x,y)=0$ and $v(x,y)$ is as above, then any solution which is $l-$conjugate to $u(x,y)\bl v(x,y)$ is given exactly as formulated in the proposition. So, we get that all these flows are obtained from $\phi(\m{x})=0\bl y$ ($c=0$) or $\phi(\m{x})=0\bl \frac{y}{y+1}$ ($c\neq 0$) by a conjugation with a $1-$BIR $\ell$. \end{proof}
Henceforth we will always assume the condition (\ref{init}).
\subsection{Level 0 flows}As we will see, the level 0 flows are in some sense the simplest ones. Here will present several chosen flows. For example,
\begin{eqnarray}
\phi_{{\rm pr}}(x,y)=\frac{x}{x+y+1}\bl \frac{y}{x+y+1}.
\label{proj}
\end{eqnarray}
\noindent {\fbox{\bf 1.}} Let $P(x,y)=xy$, $Q(x,y)=x^2+y^2$ (here and below in the setting of Proposition \ref{prop1}). Then
\begin{eqnarray*}
\phi^{(0)}_{\,1}(x,y)=\ell^{-1}_{P,Q}\circ\phi_{{\rm pr}}
\circ\ell_{P,Q}(x,y)=\frac{x(x^2+y^2)}{x^2y+xy^2+x^2+y^2}\bl\frac{y(x^2+y^2)}{x^2y+xy^2+x^2+y^2}.
\end{eqnarray*}
This function, if considered as a map, is not defined on a genus $0$ cubic $\{(x,y):x^2y+xy^2+x^2+y^2=0\}$. This curve consists of three disconnected components. Each lint through the origin intersects this curve in a single point, except the lines $x+y=0$, $x=0$ and $y=0$. If we add three open segments and three points at infinity, the solution $\phi(x,y)$ can be defined as a homeomorphism of the projective plane (see the Subsection \ref{sub5.6}).\\

\noindent {\fbox{\bf 2.}} Let $P(x,y)=x$, $Q(x,y)=y$. Then

\begin{eqnarray*}
\phi^{(0)}_{\,2}(x,y)=\ell^{-1}_{P,Q}\circ\Big{(}\frac{x}{x+1}\bl\frac{y}{x+1}\Big{)}\circ\ell_{P,Q}(x,y)=\frac{xy}{x^2+y}\bl\frac{y^2}{x^2+y}.
\end{eqnarray*}
This solution has a very simple expression, but the line $y=0$ maps to a single point. \\

\noindent {\fbox{\bf 3.}} Let $P(x,y)=xy$, $Q(x,y)=(x+y)^2$. Then
\begin{eqnarray*}
\phi^{(0)}_{\,3}(x,y)=\ell^{-1}_{P,Q}\circ\phi_{{\rm pr}}\circ\ell_{P,Q}(x,y)=\frac{x(x+y)}{xy+x+y}\bl\frac{y(x+y)}{xy+x+y}.
\end{eqnarray*}

In general, let $J(x,y)$ be a $1-$homogenic rational function. Then the function
\begin{eqnarray}
\phi(\m{x})=\frac{x}{1-J(x,y)}\bl \frac{y}{1-J(x,y)}\label{lel0}
\end{eqnarray}
is a flow, and we will later see that all level $0$ flows are given this way. Next, let $\phi_{1}(\m{x})$ and $\phi_{2}(\m{x})$ be given by (\ref{lel0}), with $J_{1}(x,y)$ and $J_{2}(x,y)$ instead of $J(x,y)$, respectively. Then a direct calculation shows that
\begin{eqnarray*}
\phi_{1}\circ\phi_{2}(\m{x})=\frac{x}{1-J_{1}(x,y)-J_{2}(x,y)}\bl \frac{y}{1-J_{1}(x,y)-J_{2}(x,y)}.
\end{eqnarray*}
So, the composition $\phi_{1}\circ\phi_{2}(\m{x})$ is again a level $0$ flow, and this shows that level $0$ flows, together with the identity flow $\phi_{{\rm id}}(\m{x})=x\bl y$, form an abelian group under composition, which corresponds to adding $J_{1}(x,y)$ and $J_{2}(x,y)$. Thus, this group is canonically isomorphic to the additive group of the field $\mathbb{R}(t)$ via
\begin{eqnarray*}
J(x,y)\mapsto J(1,t)=f(t)\in\mathbb{R}(t),\text{ the inverse is }f(t)\mapsto yf(x/y)=J(x,y).
\end{eqnarray*}

\subsection{Level $1$ flows} These flows sometimes must also be considered separately, since in several aspects they differ from level $N\geq 2$ flows; for example, as mentioned in the introduction, the set of all transformations of the form (\ref{birat}), taking a flow of level $1$ into an univariate form, is infinite. Here are basic examples:
\begin{eqnarray}
\phi_{{\rm sph},\infty}(x,y)&=&\quad(x-y)^2+x\bl(x-y)^2+y;\label{sphi}\\
\phi_{{\rm sph},1}(x,y)&=&\frac{x^2+y^2+2x}{(x+1)^2+(y+1)^2} \bl\frac{x^2+y^2+2y}{(x+1)^2+(y+1)^2}.\label{sph}\\
\phi_{{\rm tor},1}(x,y)=\frac{x}{x+1}\bl\frac{y}{y+1};&\quad& \phi_{{\rm tor},\infty}(x,y)=\quad x\bl\frac{y}{y+1}.
\label{tor}
\end{eqnarray}
In case of (\ref{sphi}), let
\begin{eqnarray*}
A_{\sigma}(x,y)=\frac{y^2}{(x-y)^2}+\sigma\cdot\frac{y}{x-y}.
\end{eqnarray*}
Then, using the notation (\ref{super}), we get
\begin{eqnarray*}
\mathcal{U}_{\sigma}(x,y)=\frac{y^2(1-\sigma)+\sigma xy+x}{(y+1)(y(1-\sigma)+\sigma x+1)}.
\end{eqnarray*}
The algorithm how to find such a family $\ell_{\sigma}(x,y)$ which, under conjugation, transfers any flow of level $1$ into an univariate form, is as follows. First we calculate the vector field of the flow $v(\phi;\m{x})=\varpi(x,y)\bl \varrho(x,y)$ (see the Subsection \ref{sub3.1} and \ref{sub3.2}). Then we solve the differential equation (\ref{differ}), which, for level $1$ flows, happens to have a family of rational solutions $f_{1}(x)+\sigma f_{0}(x)$. Finally, we set $A_{\sigma}(x,y)=f_{1}(x/y)+\sigma f_{0}(x/y)$, which is exactly the family we are looking for. For level $N\geq 2$ flows, the corresponding differential equation has a solution of the form $f_{1}(x)+\sigma f_{0}(x)$, where $f_{1}$ is rational while $f_{0}$ is algebraic but not rational. We will give an example while discussing level $2$ flows.
In case (\ref{sph}), if we set
\begin{eqnarray*}
A_{\sigma}(x,y)=\frac{2y^2}{x^2+y^2}+2\sigma\cdot\frac{xy-y^2}{x^2+y^2},
\end{eqnarray*}
we get exactly the same one parameter family $\mathcal{U}_{\sigma}(x,y)$ of flows in an univariate form. In fact, let us consider the linear transformation $L(x,y):(x,y)\mapsto(\tau x,y)$. Then for any flow $\phi(x,y)=u(x,y)\bl v(x,y)$, the flow $L^{-1}\circ\phi\circ L(x,y)$ preserves the function $v$. Thus,
\begin{eqnarray*}
\mathcal{U}_{\sigma,\tau}(x,y)=\frac{y^2(1-\sigma)+\sigma\tau xy+\tau x}{\tau(y+1)(y(1-\sigma)+\sigma\tau x+1)},\quad \sigma,\tau\in\mathbb{R},
\end{eqnarray*}
is a two parameter family of flows in an univariate form.
For convenience reasons, let us introduce
\begin{eqnarray}
\mathcal{W}^{(1)}_{\sigma,\tau}(x,y)=\mathcal{U}_{1-\sigma\tau,\tau}(x,y)=\frac{\sigma y^2+(1-\sigma\tau)xy+x}{(y+1)(\sigma\tau y+(1-\sigma\tau)\tau x+1)},\quad \sigma,\tau\in\mathbb{R}.
\label{sup1}
\end{eqnarray}
Two special cases are when $\sigma=0$ and $\tau=0$, respectively:
\begin{eqnarray*}
\mathcal{W}^{(1)}_{0,\tau}(x,y)=\frac{x}{\tau x+1};\quad \mathcal{W}^{(1)}_{\sigma,0}(x,y)=\frac{\sigma y^2+xy+x}{y+1}.
\end{eqnarray*}
There is, however, one important limit case. Indeed, let $(\sigma,\tau)=(\frac{1}{\tau}+\frac{\kappa}{\tau^2},\tau)$, $\kappa\in\mathbb{R}$. Then the direct calculation shows that
\begin{eqnarray}
\lim\limits_{\tau\rightarrow\infty}\mathcal{W}^{(1)}_{\frac{1}{\tau}+\frac{\kappa}{\tau^2},\tau}(x,y)
:=\mathcal{W}^{(1)}_{\kappa}(x,y)=\frac{x}{(y+1)(y-\kappa x+1)}.
\label{sup0}
\end{eqnarray}

The first coordinate of a vector field of the univariate flow $\mathcal{W}^{(1)}_{\sigma,\tau}(x,y)$ is
\begin{eqnarray}
\varpi^{(1)}_{\sigma,\tau}(x,y)=(\sigma\tau-1)\tau x^2-2\sigma\tau xy+\sigma y^2=Ax^2+Bxy+Cy^2.
\label{canon}
\end{eqnarray}
We see that the point $(A,B,C)$ lies on the one sheeted hyperboloid
\begin{eqnarray*}
\mathcal{H}_{1}=\{(X,Y,Z)\in\mathbb{R}^{3}: 4XZ=Y^2+2Y\}.
\end{eqnarray*}

Thus, the space of all level $1$ flows in an univariate form is canonically homeomorphic to $\mathcal{H}_{1}$. When considering level $N$ flows (in this instant,  the case $N=1$ will be no different) we will show that $\mathcal{W}^{(1)}_{\sigma,\tau}$ and $\mathcal{W}^{(1)}_{\kappa}$ can be obtained from $\phi_{N}$ (see the main Theorem) using only linear conjugation and conjugation by an involution $i(x,y)=\frac{y^2}{x}\bl y$.

 Now we will show that all basic solutions $\phi_{{\rm sph},\infty}$, $\phi_{{\rm sph},1}$, $\phi_{{\rm tor},1}$, $\phi_{{\rm tor},\infty}$ are $\ell$-conjugate. We will deeper explore this phenomenon in \cite{alkauskas2}. \\
\noindent \fbox{{\bf 1.}} Let $P(x,y)=x^{2}-y^{2}$, $Q(x,y)=xy$. Then
\begin{eqnarray*}
\ell_{P,Q}^{-1}\circ\phi_{{\rm tor},1}\circ\ell_{P,Q}(x,y)=\frac{x^2-y^2+x}{2x-2y+1}\bl\frac{x^2-y^2+y}{2x-2y+1}.
\end{eqnarray*}
The latter solution is $l-$conjugate to $\phi_{{\rm tor},\infty}$, where the linear map $L$ is given by $L(x,y)=x+\frac{y}{4}\bl x-\frac{y}{4}$. Thus,
\begin{eqnarray*}
L^{-1}\circ\ell_{P,Q}^{-1}\circ\phi_{{\rm tor},1}\circ\ell_{P,Q}\circ L=x\bl\frac{y}{y+1}=\phi_{{\rm tor},\infty}.
\end{eqnarray*}
Explicitly, the function which conjugates these two basic solutions is given by
\begin{eqnarray}
\ell_{P,Q}\circ L=\frac{4xy}{4x-y}\bl\frac{4xy}{4x+y}.
\label{ex1}
\end{eqnarray}\\

\noindent\fbox{{\bf 2.}} Let $P(x,y)=x^{2}+y^{2}$, $Q(x,y)=2xy$. Then
\begin{eqnarray*}
\phi(x,y)=\ell^{-1}_{P,Q}\circ\phi_{{\rm tor},1}\circ\ell_{P,Q}(x,y)=\phi_{{\rm sph},1}.
\end{eqnarray*}

\noindent\fbox{{\bf 3.}} Let $P(x,y)=y$, $Q(x,y)=x$. Then
\begin{eqnarray*}
\ell^{-1}_{P,Q}\circ\phi_{{\rm tor},\infty}\circ\ell_{P,Q}(x,y)=y^2+x\bl y.
\end{eqnarray*}
This solution is $l-$conjugate to $\phi_{{\rm sph},\infty}$. Indeed, if $L(x,y)=x\bl x-y$, then
\begin{eqnarray*}
L^{-1}\circ\ell_{P,Q}^{-1}\circ\phi_{{\rm tor},\infty}\circ\ell_{P,Q}\circ L=(x-y)^2+x\bl(x-y)^2+y=\phi_{{\rm sph},\infty}.
\end{eqnarray*}
 Combining with (\ref{ex1}), we see that
\begin{eqnarray}
& &\ell^{-1}\circ\phi_{{\rm tor},1}\circ\ell=\phi_{{\rm sph},\infty},\text{ where}\label{0-1}\\
\ell&=&\Big{(}\frac{4xy}{4x-y}\bl\frac{4xy}{4x+y}\Big{)}\circ\Big{(}(x-y)\bl\frac{(x-y)^2}{x}\Big{)}=\frac{4(x-y)^2}{4x-(x-y)}\bl\frac{4(x-y)^2}{4x+(x-y)}.\nonumber
\end{eqnarray}
But there is a simpler transformation which conjugates $\phi_{\rm{tor},1}$ and $\phi_{\rm{sph},\infty}$. Let $P(x,y)=(x-y)^2$, $Q(x,y)=xy$. Then
\begin{eqnarray*}
\ell^{-1}_{P,Q}\circ\phi_{\rm{tor},1}\circ\ell_{P,Q}=\phi_{\rm{sph},\infty}.
\end{eqnarray*}
Thus, conjugation with $\ell_{P,Q}\circ\ell^{-1}$ preserves $\phi_{\rm{tor},1}$. The topic of invariance of flows under the $1-$BIR will be dealt more fully in \cite{alkauskas2}.\\
\noindent\fbox{{\bf 4.}} We finish this subsection with an example of arbitrarily chosen level $1$ flow. The flow $\phi(\m{x})=\frac{x+xy+y^2}{y+1}\bl \frac{y}{y+1}$ is $l-$conjugate to $\phi_{{\rm tor},\infty}$. Let $P(x,y)=y$, $Q(x,y)=x$. Then \begin{eqnarray*}
\phi^{(1)}_{\,1}(x,y)=\ell^{-1}_{P,Q}\circ\phi\circ\ell_{P,Q}(x,y)=\frac{(x^2+y^2x+y^3)^2}{(y^2+x)x^2}\bl\frac{y(x^2+y^2x+y^3)}{x(y^2+x)}.
\end{eqnarray*}

\subsection{Level $N\geq 2$ flows} The map $i(x,y)=\frac{y^2}{x}\bl y$ is an involution (See the Appendix \ref{app}). One checks directly that
\begin{eqnarray*}
i\circ\phi_{N}\circ i=\phi_{-N},\quad N\in\mathbb{Z}.
\end{eqnarray*}
Thus, while speaking about the level we can confine to non-negative integers, since, as we see, levels $N$ and $-N$ are conjugate. (This also nicely fits to the fact that the level itself can only be define modulo a sign, since if $\mathscr{W}(u,v)={\rm const.}$ are the equations for the orbits, so are $\mathscr{W}^{-1}(u,v)={\rm const.}$; see the Subsection \ref{sub5.1}).\\

For $\phi=\phi_{N}$, the differential equation (\ref{differ}) has a general solution of the form $1+\sigma x^{-1/N}$, $\sigma\in\mathbb{R}$, so if we set
\begin{eqnarray*}
A_{\sigma}(x,y)=1+\sigma\cdot\sqrt[N]{\frac{x}{y}},
\end{eqnarray*}
and using the notation (\ref{super}), we get
\begin{eqnarray*}
\mathcal{U}_{\sigma}(x,y)=\frac{x(y+\sigma y\sqrt[N]{\frac{y}{x}}+1)^{N}}{y+1}.
\end{eqnarray*}
which is rational only if $\sigma=0$. Here we will describe some examples of level 2 flows, other levels being analogous.\\

\noindent\fbox{{\bf 1.}} Let $P(x,y)=y$, $Q(x,y)=x$. Then
\begin{eqnarray*}
\phi^{(2)}_{\,1}(x,y)=\ell_{P,Q}^{-1}\circ\phi_{N}\circ\ell_{P,Q}(x,y)=\frac{(y^{2}+x)^3}{x^2}\bl\frac{y(y^2+x)}{x}.
\end{eqnarray*}
If we are now allowed to use complex numbers, we can check that this solution is invariant under the conjugation with an order $4$ linear map $L(x,y)\mapsto(-x,iy)$. This theme is touched on the surface in the Subsections \ref{subs} and \ref{duality}; it will be deeper explored in \cite{alkauskas2}.\\

 \noindent\fbox{{\bf 2.}} Let $P(x,y)=x+y$, $Q(x,y)=y$. Then
\begin{eqnarray*}
\phi^{(2)}_{\,2}(x,y)=\ell_{P,Q}^{-1}\circ\phi_{N}\circ\ell_{P,Q}(x,y)=\frac{x(x+y+1)}{x^2+xy+2x+1}\bl\frac{y}{(x^2+xy+2x+1)(x+y+1)}.
\end{eqnarray*}

\noindent\fbox{{\bf 3.}} Let $P(x,y)=(x+y)^2$, $Q(x,y)=xy$, $L(x,y)=x\bl y-x$. Then
\begin{eqnarray*}
\phi^{(2)}_{\,3}(x,y)=L^{-1}\circ\ell_{P,Q}^{-1}\circ\phi_{N}\circ\ell_{P,Q}\circ L(x,y)=\frac{(y^2+x)^3}{(x+2xy+y^3)^2}\bl\frac{y(y^2+x)}{x+2xy+y^3}.
\end{eqnarray*}
Here we present few classes of general level $N$ flows. Let $N\in\mathbb{N}$, $N\geq 2$.
Let us consider the univariate flow whose vector field is given by
\begin{eqnarray}
\varpi^{(N)}_{\sigma,\tau}(x,y)=(\sigma\tau-N)\tau x^2+(N-1-2\sigma\tau)xy+\sigma y^2=Ax^2+Bxy+Cy^2.
\label{vecc}
\end{eqnarray}
Analogously to the case $N=1$, the point $(A,B,C)$ lies on the hyperboloid (\ref{hyper}). One can check that the vector field (\ref{vecc}) arises from the univariate flow
\begin{eqnarray}
\mathcal{W}^{(N)}_{\sigma,\tau}(x,y)=
\frac{(y+1)^N[(N-\sigma\tau)x+\sigma y ]+\sigma(\tau x-y)}{\tau(y+1)^N[(N-\sigma\tau)x+\sigma y]-(N-\sigma\tau)(\tau x-y)}\cdot\frac{y}{y+1}
,\quad \sigma,\tau\in\mathbb{R}.\label{uniN}
\end{eqnarray}
In fact, we have solved the first PDE of (\ref{parteq}) in case $\varrho(x,y)=-y^2$ and $\varpi(x,y)=\varpi^{(N)}_{\sigma,\tau}(x,y)$. The general solution to this PDE is given by
\begin{eqnarray*}
f\Bigg{(}\frac{(y+1)^N(\sigma y +(N-\sigma\tau)x)}{y-\tau x}\Bigg{)}\frac{y}{y+1},
\end{eqnarray*}
where $f$ is an arbitrary differentiable function, so the boundary condition (\ref{bound}) yields exactly $\mathcal{W}^{(N)}_{\sigma,\tau}(x,y)$. In case $N=1$, this exactly coincides with the function (\ref{sup1}). Similarly to $N=1$, for $N\geq 2$ there is one boundary case. Let the pair of indices is given by $(\frac{1}{\tau}+\frac{\kappa}{\tau^2},N\tau)$, $\kappa\in\mathbb{R}$. Then
\begin{eqnarray}
\lim\limits_{\tau\rightarrow\infty}\mathcal{W}^{(N)}_{\frac{1}{\tau}+\frac{\kappa}{\tau^2},N\tau}(x,y):=
\mathcal{W}^{(N)}_{\kappa}(x,y)=\frac{xy}{(y+1)[(y+1)^{N}(y-\kappa x)+\kappa x]}.
\label{kapa}
\end{eqnarray}
Now we will show how to obtain the flows $\mathcal{W}^{(N)}_{\sigma,\tau}$ and $\mathcal{W}^{(N)}_{\kappa}$ from the canonical flow $\phi_{N}$ using linear conjugation and conjugation by the involution $i(x,y)=\frac{y^2}{x}\bl y$. This can be done using (\ref{tiesine}), (\ref{vecconj}) and operations on vector fields; but it is also rather simple if we work directly with flows. Consider the linear map $L:(x,y)\mapsto (x+\sigma y,y)$. Then
\begin{eqnarray*}
L^{-1}\circ\phi_{N}\circ L(x,y)=\frac{(x+\sigma y)(y+1)^N-\sigma y}{y+1}\bl \frac{y}{y+1},
\end{eqnarray*}
which is a flow. Conjugating this with the involution $i$, we get
\begin{eqnarray*}
i\circ L^{-1}\circ\phi_{N}\circ L\circ i(x,y)=\frac{xy}{(y+1)[(y+1)^{N}(y+\sigma x)-\sigma x]}\bl\frac{y}{y+1}.
\end{eqnarray*}
This flow is exactly $\mathcal{W}^{(N)}_{-\sigma}(x,y)$, given by (\ref{kapa}); the set obtained by conjugating these flows with linear maps of the form $(x,y)\mapsto(ax+by,y)$ cover the whole family $\mathcal{W}^{(N)}_{\sigma,\tau}$.\\
 
 Therefore, for any flow $\phi$ we have given an explicit algorithm how to find its level and the $1-$homogenic transformation $\ell$ such that the identity in the Theorem holds: we first find $\ell$ such that $\ell^{-1}\circ\phi\circ\ell$ is in a univariate form. This amounts to solving the differential equation (\ref{differ}). Then we use the above conjugations to transform it into the canonical solution $\phi_{N}$.\\

\begin{table}[h]
\begin{tabular}{|r | r|r| r| r|}
\hline
\multicolumn{5}{|c|}{\textbf{Selected $2$-dimensional projective flows}}\\
\hline
$\phi(\m{x})$ & $\mathscr{W}(\phi;u,v)$ & $v(\phi;x,y)$ & $ $ & Zeros-poles\\
\hline\hline
\emph{Level} 0 & $ $& $ $ & $ $ & $ $\\
$\phi_{{\rm pr}}(\m{x})$ & $uv^{-1}$& $-x(x+y)\bl-y(x+y)$ &- & $(1;0)$\\
$\phi^{(0)}_{\,1}(\m{x})$ & $uv^{-1}$& $-\frac{x^2y(x+y)}{x^2+y^2}\bl-\frac{xy^{2}(x+y)}{x^2+y^2}$& - & $(3;0)$\\
$\phi^{(0)}_{\,2}(\m{x})$ & $uv^{-1}$ & $-x^3y^{-1}\bl(-x^2) $& - & $(2;1)$\\
$\phi^{(0)}_{\,3}(\m{x})$ & $uv^{-1}$ & $-x^2y(x+y)^{-1}\bl -y^2x(x+y)^{-1} $& - & $(2;1)$\\
\hline\hline
\emph{Level} 1 & $ $& $ $ & $\widehat{p}(\phi)$ &  $ $\\
$\phi_{{\rm sph},\infty}(\m{x})$ & $u-v$& $(x-y)^2\bl (x-y)^2$ & 1& $(2;0)$\\
$\phi_{{\rm sph},1}(\m{x})$ & $(u^2+v^2)(u-v)^{-1}$& $-\frac{1}{2}x^2+\frac{1}{2}y^2-xy\bl \frac{1}{2}x^2-\frac{1}{2}y^2-xy$& 1& $(0;0)$\\
$\phi_{{\rm tor},\infty}(\m{x})$ & $v$& $0\bl(-y^2)$& 0 & $(2;0)$\\
$\phi_{{\rm tor},1}(\m{x})$ & $uv(u-v)^{-1}$& $(-x^2)\bl(-y^2) $& 1& $(0;0)$\\
$\phi^{(1)}_{\,1}(\m{x})$ & $(u+v)vu^{-1}$ & $y^2(x+2y)x^{-1}\bl y^4x^{-2}$& 0 & $(2;2)$\\
$\Psi(\m{x})$ & $uv(u-v)^{-1}$ & $\frac{x^2(x+y)^2}{(x-y)^2}\bl \frac{y^2(x+y)^2}{(x-y)^2}$& $-1$& $(2;2)$\\
$\Phi_{1}(\m{x})$ & $(u+v)^2(u-v)^{-1}$ & $-\frac{3}{2}x^2-xy+\frac{1}{2}y^2\bl \frac{1}{2}x^2-xy-\frac{3}{2}y^2$ & 1& $(1;0)$\\
$\Phi'_{1}(\m{x})$ & $u-v$ & $-\frac{1}{2}(x+y)^2\bl-\frac{1}{2}(x+y)^2$& $-1$& $(2;0)$\\
$\phi_{-1}(\m{x})$ & $v^2u^{-1}$ & $-2xy\bl-y^2$& $\infty$& $(1;0)$\\
\hline\hline
\emph{Level} $2$ & $ $& $ $ & $p_{2}(\phi)$ & $ $\\
$\phi_{2}(\m{x})$ & $uv$& $xy\bl(-y^2)$ & (0,1,0)& $(1;0)$\\
$\phi^{(2)}_{\,1}(\m{x})$ & $v^{3}u^{-1}$& $3y^2\bl y^{3}x^{-1}$& (0,1,0)& $(2;1)$\\
$\phi^{(2)}_{\,2}(\m{x})$ & $(u+v)^{2}uv^{-1}$& $-x^2+xy\bl-3xy-y^2$& (0,1,0)& $(0;0)$\\
$\phi^{(2)}_{\,3}(\m{x})$ & $v^{4}u^{-1}(u-v)^{-1}$& $-4xy+3y^2\bl (y^3-2xy^2)x^{-1} $& (-2,1,0)& $(1;1)$\\
$\Phi_{2}(\m{x})$ & $(u+v)^3(u-v)^{-1}$ & $-2x^2-xy+y^2\bl x^2-xy-2y^2 $&(-1,-1,1)& $(1;0)$\\
\hline
\end{tabular}
\caption{The column ``zeros-poles" refers to how many zeros (where both coordinates vanish) or poles (where at least one coordinate is infinity) does the vector field of a flow has on a half of a unit circle, counting with multiplicities.}
\label{table1}
\end{table}
\begin{figure}[h]
\centering
\begin{tabular}{c c}
\epsfig{file=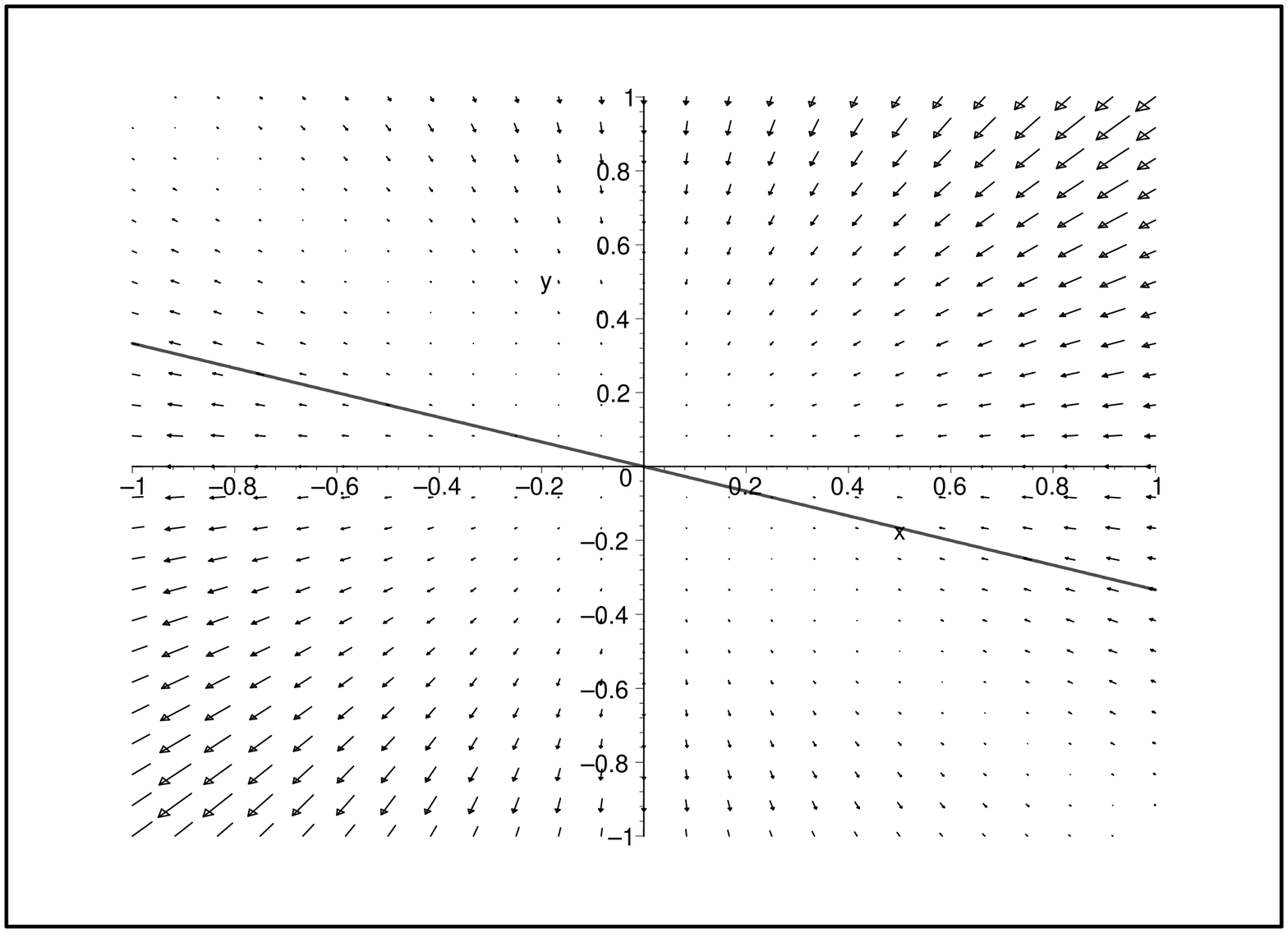,width=228pt,height=228pt,angle=-90}
&\epsfig{file=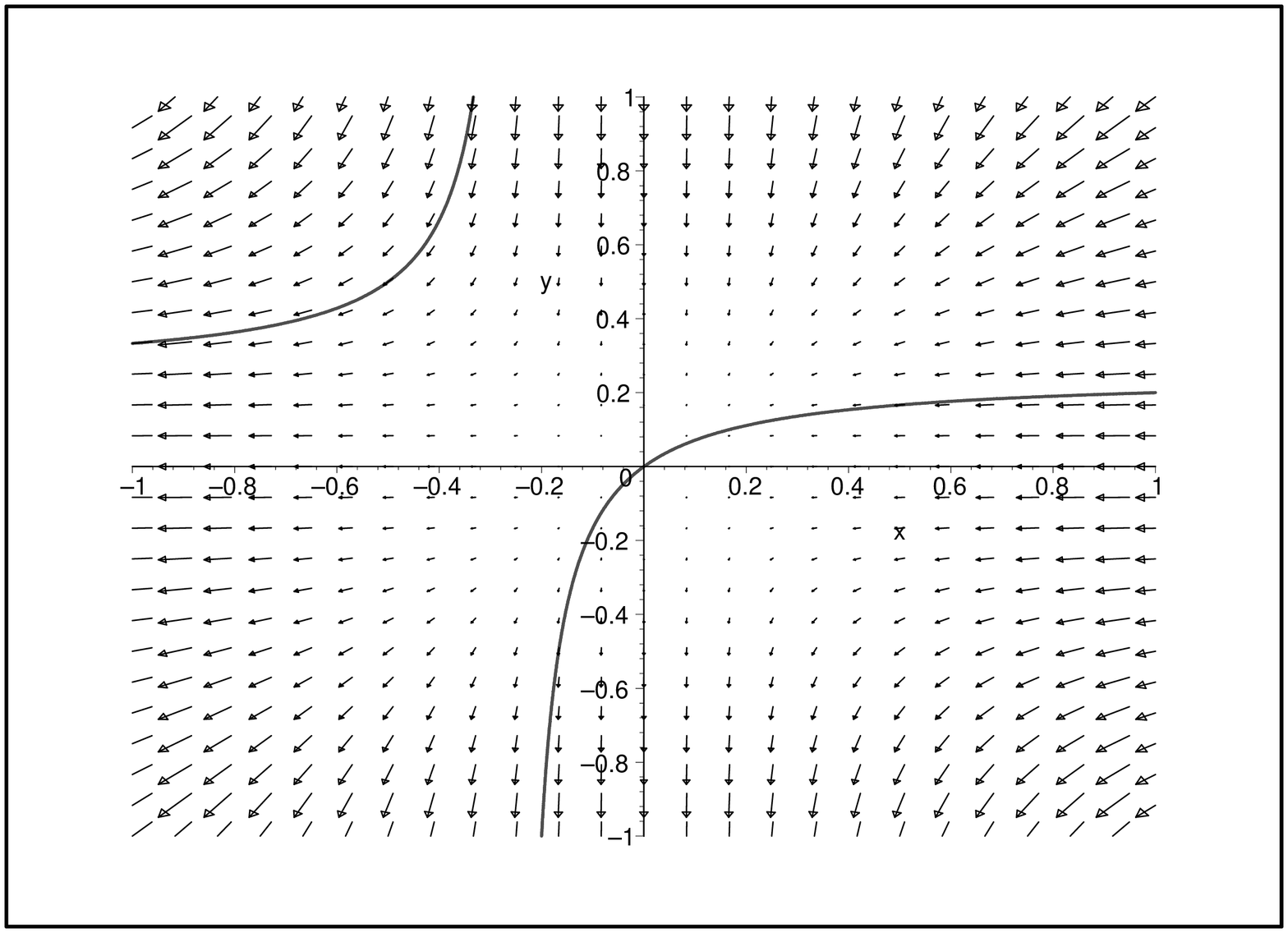,width=228pt,height=228pt,angle=-90} \\
$v(\phi_{{\rm pr}};\m{x})$, $\mathscr{W}(\phi;u,v)=-3$.& $v(\phi_{{\rm tor},1};\m{x})$, $\mathscr{W}(\phi;u,v)=0.25$.\\
\epsfig{file=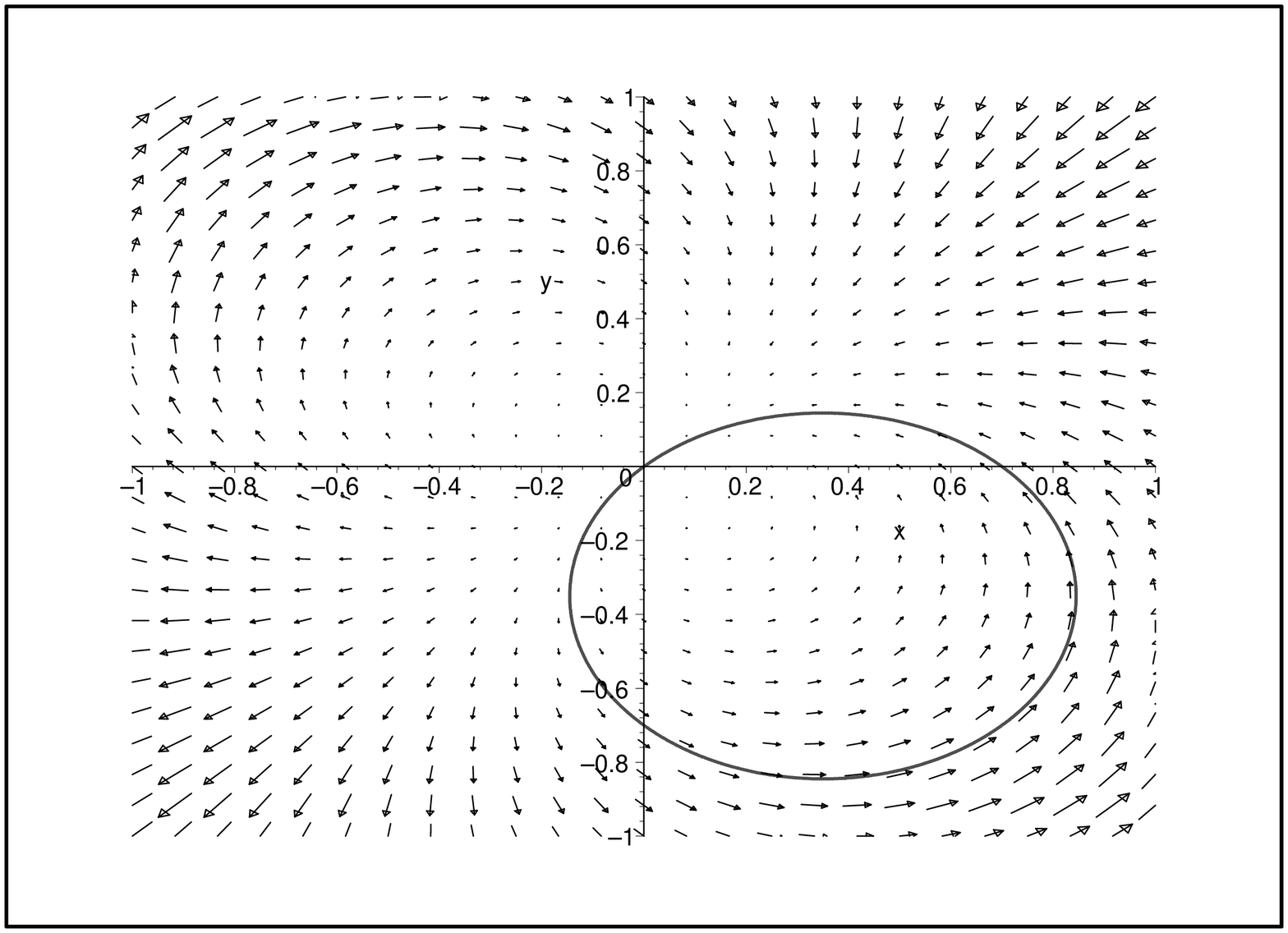,width=228pt,height=228pt,angle=-90} &
\epsfig{file=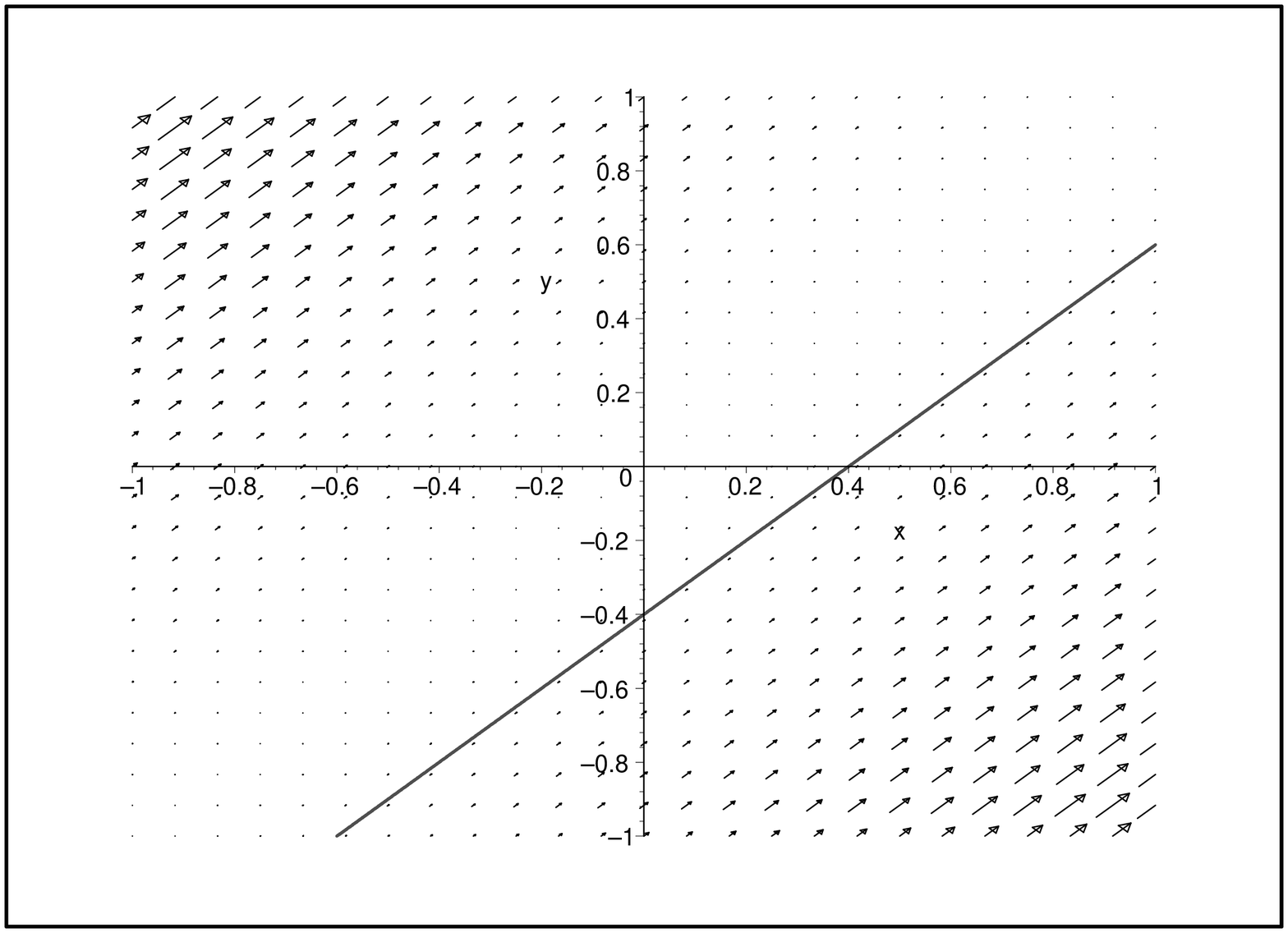,width=228pt,height=228pt,angle=-90}\\
$v(\phi_{{\rm sph},1};\m{x})$, $\mathscr{W}(\phi;u,v)=0.7$. & $v(\phi_{{\rm sph},\infty};\m{x})$, $\mathscr{W}(\phi;u,v)=0.4$.
\end{tabular}
\caption{Vector fields $v(\phi;\m{x})$ of basic flows with a selected orbit}.
\end{figure}
\begin{figure}[h]
\epsfig{file=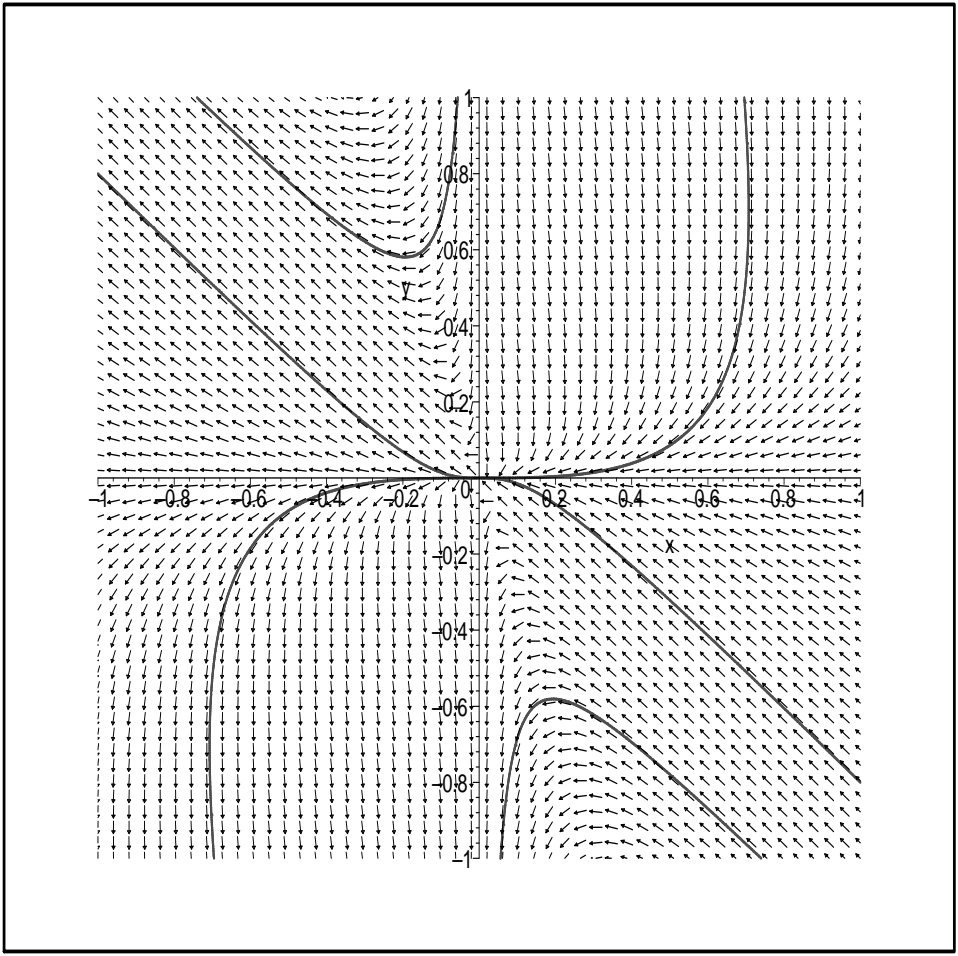,width=360pt,height=360pt,angle=-90}
\caption{(Normalized) vector field $v(\phi^{(2)}_{\,2};\m{x})$ with two selected orbits $\mathscr{W}=-0.05$ (three branches in quarters II and IV) and $\mathscr{W}=2$ (one branch in quarters I and III).}.
\end{figure}
\begin{figure}[h]
\epsfig{file=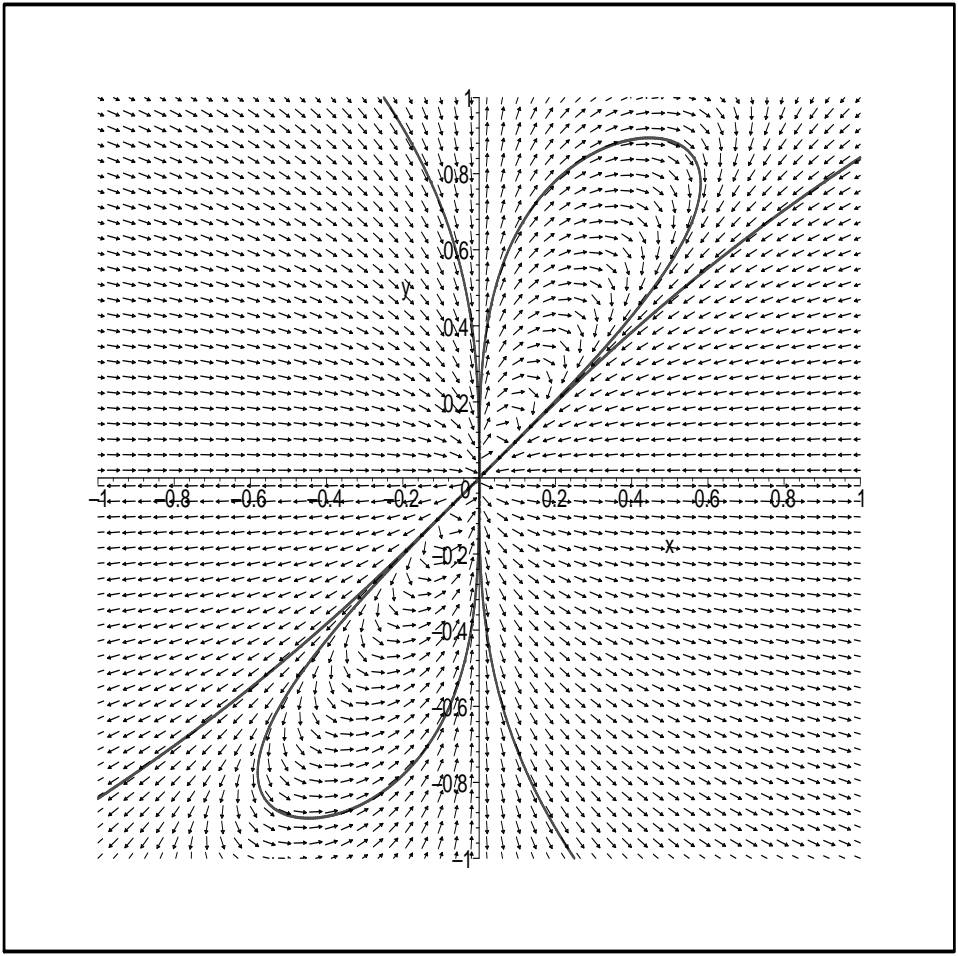,width=360pt,height=360pt,angle=-90}
\caption{(Normalized) vector field $v(\phi^{(2)}_{\,3};\m{x})$ with two selected orbits $\mathscr{W}=3.2$ ($``8"$ form) and $\mathscr{W}=-3.2$ (two branches).}.
\end{figure}
\section{Vector field of a rational flow and PDEs}\label{sec3}
\subsection{Basic properties of a vector field}\label{sub3.1}If $\phi$ is a rational solution to (\ref{funk}), then, as already defined in the introduction, each point possesses an orbit
\begin{eqnarray*}
\m{x}_{z}=\phi^{z}(\m{x})=\frac{1}{z}\cdot\phi(z\m{x}),\quad z\in\mathbb{R}\cup\{\infty\};\quad\m{x}_{0}=\m{x},\quad 
\bigcup\limits_{z}\m{x}_{z}=\mathscr{V}(\m{x}_{0}).
\end{eqnarray*}
A \emph{vector field} of the flow is defined  by
\begin{eqnarray}
v(\phi;\m{x})=\lim\limits_{z\rightarrow 0}\frac{\phi^{z}(\m{x})-\m{x}}{z}=\frac{\d}{\d z}\phi^{z}(\m{x})\Big{|}_{z=0}.
\label{vector}
\end{eqnarray}
This is obviously a rational function; we will denote its coordinates by $v(\phi;\m{x})=\varpi(x,y)\bl \varrho(x,y)$.
\begin{prop}
A function $v(\phi;\m{x})$ is a pair of homogenic functions of degree $2$:
\begin{eqnarray*}
v(\phi;a\m{x})=a^{2}v(\phi;\m{x}),\text{ for each }a\in\mathbb{R},\m{x}\in\mathbb{R}^{2}.
\end{eqnarray*}
\end{prop}
\begin{proof}Indeed, this is obvious if $a=0$. If $a\neq 0$, then
\begin{eqnarray*}
v(\phi;a\m{x})=\lim\limits_{z\rightarrow 0}\frac{\phi(za\m{x})/z-a\m{x}}{z}=a^2\lim\limits_{z\rightarrow 0}\frac{\phi(za\m{x})/(az)-\m{x}}{az}=
a^{2}v(\phi;\m{x}).
\end{eqnarray*}
\end{proof}
The latter property is essentially using the fact that our translation equation arises from a conjugation by a homothety. Of course, a vector field does not define a function uniquely, even if we confine to rational functions. For example, if $\psi(\m{x})=x(x+1)^{-a}\bl y(y+1)^{-b}$ and $\phi(\m{x})=x(ax+1)^{-1}\bl y(by+1)^{-1}$, then $v(\psi;\m{x})=(-ax^{2})\bl(-by^{2})=v(\phi;\m{x})$, but $\psi$ defines a flow only for $a=b=1$, while $\phi$ is always a flow. We see that when $a,b\in\mathbb{Z}$, two pairs of rational functions might have the same vector field. The vector $v(\phi,\m{x}_{0})$ is tangent to the orbit $\mathscr{V}(\m{x}_{0})$ for any $\m{x}_{0}$ provided that $\mathscr{V}(\m{x}_{0})$ is not a single point.\\
\indent Now we will explicitly calculate some vector fields. The next crucial proposition shows the properties of the vector field under conjugation.
\begin{prop}Suppose, $L$ is a non-degenerate linear map, and $\phi(\m{x})$ is a flow. Then
\begin{eqnarray}
v\Big{(}L^{-1}\circ\phi\circ L;\m{x}\Big{)}=L^{-1}\circ v(\phi;\m{x})\circ L.
\label{tiesine}
\end{eqnarray}
Further, let a birational map $\ell_{P,Q}$ be given by (\ref{birat}), $A(x,y)=P(x,y)Q^{-1}(x,y)$, which is a $0-$homogeneous function. Suppose that $v(\phi,\m{x})=\varpi(x,y)\bl \varrho(x,y)$. Then
\begin{eqnarray}
v(\ell^{-1}_{P,Q}\circ\phi\circ\ell_{P,Q};\m{x})&=&\nonumber\\
\varpi'(x,y)\bl\varrho'(x,y)&=&\nonumber\\
A(x,y)\varpi(x,y)-A_{y}[x\varrho(x,y)-y\varpi(x,y)]&\bl&\label{vecconj}\\
A(x,y)\varrho(x,y)+A_{x}[x\varrho(x,y)-y\varpi(x,y)].&&\nonumber
\end{eqnarray}
As a corollary,
\begin{eqnarray*}
x\varrho'(x,y)-y\varpi'(x,y)=A(x,y)[x\varrho(x,y)-y\varpi(x,y)].
\end{eqnarray*}

\label{conjug}
\end{prop}
\begin{proof}
The first part is established by a direct calculation. Further, the first coordinate of $\ell^{-1}_{P,Q}\circ\phi\circ\ell_{P,Q}(x,y)$ is given by (here $A^{-1}=\frac{1}{A}$)
\begin{eqnarray*}
\widehat{u}(x,y)=u\big{(}xA(x,y),yA(x,y)\big{)}\cdot A^{-1}\Big{[}u\big{(}xA(x,y),yA(x,y)\big{)},v\big{(}xA(x,y),yA(x,y)\big{)}\Big{]}.
\end{eqnarray*}
Thus, since $A(x,y)$ is $0-$homogeneous,
\begin{eqnarray*}
&&\frac{\widehat{u}(xz,yz)}{z}=\\
&&z^{-1}u\big{(}xzA(x,y),yzA(x,y)\big{)}\cdot A^{-1}\Big{[}u\big{(}xzA(x,y),yzA(x,y)\big{)},v\big{(}xzA(x,y),yzA(x,y)\big{)}\Big{]}.
\end{eqnarray*}
Let us take the derivative with respect to $z$ and substitute $z=0$. Using (\ref{init}) and (\ref{vector}), we get
\begin{eqnarray*}
\varpi'(x,y)=\frac{\d}{\d z}\frac{\widehat{u}(xz,yz)}{z}\Big{|}_{z=0}=A(x,y)\varpi(x,y)-xA_{x}\varpi(x,y)-xA_{y}\varrho(x,y).
\end{eqnarray*}
Similarly,
\begin{eqnarray*}
\varrho'(x,y)=\frac{\d}{\d z}\frac{\widehat{v}(xz,yz)}{z}\Big{|}_{z=0}=A(x,y)\varrho(x,y)-yA_{x}\varpi(x,y)-yA_{y}\varrho(x,y).
\end{eqnarray*}
Since $A$ is a $0-$homogenic function, we have $xA_{x}+yA_{y}\equiv 0$, and this implies the statement of the proposition.
\end{proof}

After having proved the main Theorem we will know \emph{a posteriori} that for a given vector field $\varpi(x,y)\bl \varrho(x,y)$ of a rational flow $\phi(\m{x})$ there exists a $0-$homogenic rational function $A$ such that $\varrho'(x,y)=-y^2$, or, in other words, the flow is in an univariate form; thus, according to (\ref{vecconj}),
\begin{eqnarray*}
A(x,y)\varrho(x,y)+A_{x}(x,y)\Big{(}x\varrho(x,y)-y\varpi(x,y)\Big{)}=-y^{2}.
\end{eqnarray*}
Let $f(x)=A(x,1)$, $\varpi(x)=\varpi(x,1)$, $\varrho(x)=\varrho(x,1)$. These are all rational functions. Thus, there exits a rational function $f\in\mathbb{R}(x)$ such that
\begin{eqnarray}
f(x)\varrho(x)+f'(x)(x\varrho(x)-\varpi(x))=-1.\label{differ}
\end{eqnarray}
We will frequently refer to this differential equation while performing practical tasks; for example, it is very useful in finding explicitly the function $\ell$ whose existence is guaranteed by the Theorem. This will also be crucial in the final step of the proof of the main Theorem, since it provides the bridge between the flows over $\mathbb{C}$ and flows over $\mathbb{R}$.
\subsection{Partial differential equations}
\label{sub3.2}
Suppose $\phi(\m{x})$ satisfies (\ref{funk}). Writing this equation explicitly, we get
\begin{eqnarray*}
u\Big{(}\frac{1-z}{z}\cdot u(xz,yz),\frac{1-z}{z}\cdot v(xz,yz)\Big{)}=(1-z)u(x,y),
\end{eqnarray*}
and the same holds for the function $v$. Now take the full derivative with respect to $z$, substitute $z=0$ afterwards, and use (\ref{vector}). We obtain
\begin{eqnarray}
\left\{\begin{array}{c@{\qquad}l}
u_{x}(x,y)(\varpi(x,y)-x)+
u_{y}(x,y)(\varrho(x,y)-y)=-u(x,y),\\
v_{x}(x,y)(\varpi(x,y)-x)+
v_{y}(x,y)(\varrho(x,y)-y)=-v(x,y).
\end{array}\right.\label{parteq}
\end{eqnarray}
The boundary conditions are as (\ref{init}):
\begin{eqnarray}
\lim_{z\rightarrow 0}\frac{u(xz,yz)}{z}=x,\quad
\lim_{z\rightarrow 0}\frac{v(xz,yz)}{z}=y.
\label{bound}
\end{eqnarray}
Solving (\ref{parteq}) with respect to $\varpi$ and $\varrho$ yields
\begin{eqnarray}
\left\{\begin{array}{c@{\qquad}l}\varpi(x,y)=\displaystyle{\frac{vu_{y}-uv_{y}}{u_{x}v_{y}-u_{y}v_{x}}}+x,\\
\\
\varrho(x,y)=\displaystyle{\frac{uv_{x}-vu_{x}}{u_{x}v_{y}-u_{y}v_{x}}}+y.
\end{array}\right.\label{vfield}
\end{eqnarray}
This gives alternative to (\ref{vector}) to calculate the vector field.\\

 The general strategy to solve linear PDE of the form (\ref{parteq}) shows us that we need to solve the following two characteristic ODEs:
\begin{eqnarray}
\frac{\d y}{\d x}&=&\frac{\varrho(x,y)-y}{\varpi(x,y)-x};\label{ordeq1}\\
\frac{\d z}{\d y}&=&\frac{z}{y-\varrho(x,y)}.\label{ordeq2}
\end{eqnarray}

In another direction, we will show that, subject to certain conditions (see Proposition \ref{ppp}), some pairs of $2-$homogenic rational functions $\varpi(x,y)\bl\varrho(x,y)$ gives rise to the unique rational flow $\phi(\m{x})$. We see from (\ref{parteq}) that both functions $u$ and $v$ satisfy the same first order linear PDE, only with different boundary conditions (\ref{bound}).\\

We will now proceed with the proof of Proposition \ref{propPDE}.
\begin{proof}Suppose $\phi$ is a flow. Then $u,v$ satisfies
(\ref{parteq}) for certain $2-$homogenic rational functions $\varpi(x,y)$ and $\varrho(x,y)$. The equations
\begin{eqnarray}
x\varpi_{x}+y\varpi_{y}=2\varpi,\quad x\varrho_{x}+y\varrho_{y}=2\varrho,
\label{inv}
\end{eqnarray}
which are automatically satisfied by $2-$homogenic functions, written in terms of $u$ and $v$ via (\ref{vfield}), after certain straightforward calculations, yield the system in Proposition \ref{propPDE}. In the opposite direction, suppose $u,v$ satisfy the system in Proposition \ref{propPDE}. Let us define $\varpi$ and $\varrho$ by (\ref{vfield}). Then the system in consideration can be written as (\ref{inv}). This shows that both $\varpi$ and $\varrho$ are $2-$homegenic rational functions. And so, we are finally left to show that the rational solution of the system (\ref{parteq}) with initial conditions (\ref{bound}), where $\varpi$ and $\varrho$ are $2-$homogeneous - call this solution $\phi(x,y)=u(x,y)\bl v(x,y)$ - yields a flow.\\
\indent The first equation in (\ref{parteq}), after a substitution $(x,z)\mapsto (xz,yz)$ can be rewritten as
\begin{eqnarray}
z^{-2}(u_{x}(xz,yz)xz+u_{y}(xz,yz)yz-u(xz,yz))=u_{x}(xz,yz)\varpi(x,y)+u_{y}(xz,yz)\varrho(x,y).
\label{homo}
\end{eqnarray}
We will now show how this equation can help us to find formally the solution to (\ref{parteq}) without integrating it. The function $u(xz,yz)$ can be (also formally) expanded into the powers of $z$:
\begin{eqnarray}
u(xz,yz)=xz+\sum\limits_{i=2}^{\infty}z^{i}\varpi^{(i)}(x,y);\label{seru}
\end{eqnarray}
here $\varpi^{(i)}(x,y)$ is a certain $i-$homogenic rational function. Powers of $z$ with exponents $i\leq 0$ are not present because of (\ref{bound}). One can find the functions $\varpi^{(i)}$ explicitly. In fact, we have
\begin{eqnarray*}
u_{x}(xz,yz)&=&1+\sum\limits_{i=2}^{\infty}z^{i-1}\varpi_{x}^{(i)}(x,y);\\
u_{y}(xz,yz)&=& \sum\limits_{i=2}^{\infty}z^{i-1}\varpi_{y}^{(i)}(x,y).
\end{eqnarray*}
Remember that $x\varpi^{(i)}_{x}(x,y)+y\varpi^{(i)}_{y}(x,y)=i\varpi^{(i)}(x,y)$. Thus, we plug the above into (\ref{homo}) and compare the coefficients at equal powers of $z$. This yields
$\varpi^{(2)}(x,y)=\varpi(x,y)$, and also the recurrence
\begin{eqnarray}
\varpi^{(i+1)}(x,y)=\frac{1}{i}\,[\varpi^{(i)}_{x}(x,y)\varpi(x,y)+\varpi^{(i)}_{y}(x,y)\varrho(x,y)],\quad i\geq 2.
\label{upow}
\end{eqnarray}
This also holds for $i=1$ if we make a natural convention that $\varpi^{(1)}(x,y)=x$.
In a similar fashion, if
\begin{eqnarray*}
v(xz,yz)=yz+\sum\limits_{i=2}^{\infty}z^{i}\varrho^{(i)}(x,y),
\end{eqnarray*}
then
\begin{eqnarray}
\varrho^{(i+1)}(x,y)=\frac{1}{i}\,[\varrho^{(i)}_{x}(x,y)\varpi(x,y)+\varrho^{(i)}_{y}(x,y)\varrho(x,y)],\quad i\geq 1,\quad \varrho^{(1)}(x,y)=y.
\label{vpow}
\end{eqnarray}
\begin{Note}\label{note3} The formal operators of the shape (\ref{upow}) were considered in the relation to P\'{o}lya-Eggenberger urn model. Suppose there is an urn containing balls of two types: ${\sf x}$ and ${\sf y}$. At each epoch, we choose one ball at random, and then add or take balls according to the result of this random pick. We refer the reader to \cite{flajolet} and the Subsection 2.4, \cite{alkauskas-un} for more details about this unexpected link. The classical urn model deals with the function $\varpi(x,y)$ only of the form $x^{a+s}y^{-a}$, $a\in\mathbb{Z}$; $s$ here is $2$.  Combinatorically this means that if we choose an ${\sf x}$ ball, we replace it by $a+s$ balls of type ${\sf x}$ and $-a$ balls of type ${\sf y}$; negative numbers stands for the subtraction. A more general urn model deals with the case of arbitrary $s$, not just $s=2$. On the other hand, $\varpi(x,y)$ in our setting is just any $2-$homogenic function, and, unless it is a monomial, it does not have a simple combinatorial interpretation. Finally, note that the term ``projective" in case of flows corresponds to the term ``balanced" in urn theory. Thus, the two theories (projective flows and urn models) have a nonempty intersection, and this is exactly the vector fields where each coordinate is a homoegeneous degree $2$ monomial.\end{Note}

Next, let us define
\begin{eqnarray*}
\Omega^{u}(x,y,z,w)=\frac{1}{z}\,u\Big{(}\frac{z}{w}\,u(xw,yw),\frac{z}{w}\,v(xw,yw)\Big{)}
-\frac{1}{z+w}u((z+w)x,(z+w)y).
\end{eqnarray*}
This function is properly defined for (almost) all $x,y,z,w$, since we are dealing with rational functions $u$ and $v$ and thus these can be iterated without restrictions on the arguments. It can be checked directly that the identity (\ref{homo}) is equivalent to
\begin{eqnarray*}
\frac{\p^{i}}{\p w^{i}}\,\Omega^{u}(x,y,z,w)\Big{|}_{w=0}\equiv 0
\end{eqnarray*}
for $i=1$. Moreover, the two recursions (\ref{upow}) and (\ref{vpow}) state precisely that this is also true for $i\geq 2$. Therefore, $\Omega^{u}(x,y,z,w)$ is independent of $w$. Finally, the boundary condition (\ref{bound}) gives
\begin{eqnarray*}
\Omega^{u}(x,y,z,w)=\Omega^{u}(x,y,z,0)\equiv 0.
\end{eqnarray*}
The same is true for an analogous function $\Omega^{v}(x,y,z,w)$. This exactly means that $u\bl v$ is a flow. This implies also the Corollary \ref{corr2}.
\end{proof}
We will make the first simplification of the equation (\ref{ordeq1}) in case $\varpi$ and $\varrho$ are two arbitrary $2$-homogenic functions. This equation has the form
\begin{eqnarray*}
y'=\frac{x^{2}\varrho(1,y/x)-y}{x^{2}\varpi(1,y/x)-x}.
\end{eqnarray*}
Let $y=wx$, where $w$ is a function in $x$, and let $\varpi(1,w)=A(w)$, $\varrho(1,w)=B(w)$. Then the above equation can be written as
\begin{eqnarray*}
B(w)-wA(w)=xA(w)w'-w'.
\end{eqnarray*}
Now consider $w$ as a variable and $x$ as a function in $w$. This gives the linear non-homogeneous ODE
\begin{eqnarray}
(wA(w)-B(w))x'+A(w)x=1.
\label{ODE}
\end{eqnarray}
Suppose that $wA(w)-B(w)\equiv 0$; that is, suppose that $\varpi(x,y)=xJ(x,y)$, $\varrho(x,y)=yJ(x,y)$, where $J(x,y)$ is a $1-$homogenic rational function. The two equation (\ref{ordeq1}) and (\ref{ordeq2}) can be easily integrated, and the general solution of the PDE (\ref{parteq}) is given by
\begin{eqnarray*}
z=\frac{yf(x/y)}{1-J(x,y)},
\end{eqnarray*}
where $f$ is an arbitrary differentiable function. Thus, the solution of (\ref{parteq}) which satisfy (\ref{bound}) are given by
\begin{eqnarray*}
u(x,y)=\frac{x}{1-J(x,y)},\quad v(x,y)=\frac{y}{1-J(x,y)}.
\end{eqnarray*}
We see that this is indeed a flow with the vector field $xJ(x,y)\bl yJ(x,y)$, where $J$ can be an arbitrary $1-$homogenic rational function. This is a level $0$ flow. In general case, we will refer to the equation (\ref{ODE}) several times.

\section{The proof of the Theorem}\label{secfin}
With all these preliminary results, we are now in position to finish the proof of the main Theorem. First, we will prove all our results for flows over the field $\mathbb{C}$. The plan is as follows.
\begin{itemize}
\item[I.](Reduction). Given any rational flow $\phi$ whose vector field is a pair of $2-$homogenic rational functions
$v(\phi;\m{x})=\varpi(x,y)\bl\varrho(x,y)$, where $\varpi$ and $\varrho$ have a common denominator of degree $d$. We prove that, unless we encounter an obstruction (the meaning of this will be explained shortly), there exists a $1-$BIR $\ell$ of a simple form, such that $\ell^{-1}\circ v(\phi;\m{x})\circ\ell=\varpi'\bl\varrho'$ with the lowered degree of the common denominator of $\varpi'$ and $\varrho'$. Our main tool are the two identities (\ref{tiesine}) and (\ref{vecconj}). Here we will not use the fact that $\phi$ is given by rational functions; it suffices that the vector field is.
\item[II.] (Obstruction). This means that in the course of reduction, which is done recurrently, we encounter the case where $\varpi$ and $\varrho$ are proportional; or, after a conjugation with a linear change, we see that the obstruction arises when $\varrho=0$. In this case the degree of the common denominator of $\varpi'\bl\varrho'=\ell^{-1}\circ(\varpi,0)\circ\ell$ cannot be lowered via a conjugation with any $1-$BIR $\ell$.  Thus, in this step we solve the system of PDEs (\ref{parteq}) with exactly the assumption $\varrho=0$. It appears that this has a rational solution only for level $1$ flow.
\item[III.]If we do not encounter an obstruction, after the reduction we arrive to the state where both $\varpi$ and $\varrho$ are quadratic forms. In this step we deal with this situation. We only exclude the most interesting cases when $\varpi$ (after the linear conjugation) is proportional to $x^2$ or $\varrho$ is proportional to $y^2$, since this will be the topic of the step IV.

\item[IV.]In this step we will prove that if $\varpi(x,y)$ is a quadratic form, and $\varrho(x,y)=-y^2$, then there exists a natural condition on the coefficients of $\varpi$ in order it to arise from a rational flow, and all the solutions are essentially given by the univariate flows $\mathcal{W}^{(N)}_{\sigma,\tau}$ and $\mathcal{W}^{(N)}_{\kappa}$, see (\ref{uniN}) and (\ref{kapa}).
\item[V.]Finally, we have shown just after the equation (\ref{kapa}) that a vector field $Ux^2+Vxy+Wy^2\bl(-y^2)$, with $(V+1)^2-4UW=N^2$, $N\in\mathbb{N}$, with the help of linear maps and the involution $i(x,y)=\frac{y^2}{x}\bl y$ can be transformed into $(N-1)xy\bl(-y^2)$. This is the vector field of the flow $\phi_{N}$, and this proves the Theorem over $\mathbb{C}$.
\item[VI.]We have seen that for a vector field arising from a rational flow, the differential equation (\ref{differ}) has a \emph{complex} rational solution, and with its help we can transform the flow into an univariate form. But since $\varrho$ and $\varpi$ are real, the real part of the solution $f(x)$ also solves this equation. So, we can transform the flow $\phi$ into the univariate form using conjugation with \emph{real} $1$-BIR function, and this proves the Theorem over $\mathbb{R}$.
\end{itemize}
\subsection{Step I }
\begin{prop}
Assume that
\begin{eqnarray*}
\varpi(x,y)=\frac{P(x,y)}{D(x,y)}\quad \varrho(x,y)=\frac{Q(x,y)}{D(x,y)},
\end{eqnarray*}
where $P,Q,D$ are homogeneous polynomials, $d=\deg(P)=\deg(Q)=\deg(D)+2\geq 3$. Moreover, assume that $P$ and $Q$ are not proportional. Then there exists a sequence of alternating linear change (\ref{tiesine}) and conjugation (\ref{vecconj}) by $\ell(x,y)=xA(x,y)\bl yA(x,y)$,  where $A$ is a linear-fractional $0-$homogenic function, such that the new vector field $\varpi'(x,y)\bl\varrho'(x,y)$, given by (\ref{vecconj}), is a pair of $2-$homogenic functions
\begin{eqnarray*}
\varpi'(x,y)=\frac{P'(x,y)}{D'(x,y)}\quad \varrho'(x,y)=\frac{Q'(x,y)}{D'(x,y)},
\end{eqnarray*}
with the lowered degree in denominator: $\deg(D')\leq \deg(D)-1$.\\
\
\end{prop}
Of course, If needed, we can always achieve that $\varpi$ and $\varrho$ have the same denominator, by bringing to the common denominator in case of necessity.
\begin{proof}
 The identities (\ref{vecconj}) can be rewritten as
\begin{eqnarray}
\varpi'(x,y)=A\frac{P}{D}-A_{y}\frac{xQ-yP}{D},\quad \varrho'(x,y)=A\frac{Q}{D}+A_{x}\frac{xQ-yP}{D}.
\label{perein}
\end{eqnarray}
If $xQ-yP\equiv 0$, this gives level $0$ flow, and this case was already settled. The corollary just after (\ref{vecconj}) shows that if this is not the case, then $x\varrho'-y\varpi'$ is not identically $0$. To start the proof, let us search for $A$ of the form
\begin{eqnarray*}
A(x,y)=\frac{xy_{0}-yx_{0}}{x\chi-y\xi},\quad x_{0},y_{0},\xi,\chi\in\mathbb{C},\quad x_{0}\chi-y_{0}\xi\neq 0.
\end{eqnarray*}
We choose the numerator $(xy_{0}-yx_{0})$ in such a way that $(xy_{0}-yx_{0})$, as a linear polynomial, divides $D(x,y)$. Further,
\begin{eqnarray*}
A_{x}=\frac{y(x_{0}\chi-y_{0}\xi)}{(x\chi-y\xi)^{2}},\quad
A_{y}=-\frac{x(x_{0}\chi-y_{0}\xi)}{(x\chi-y\xi)^{2}}.
\end{eqnarray*}
Minding (\ref{perein}), we can achieve the decrease by $1$ in the degree of denominator of $\varpi'$ and $\varrho'$ if both
\begin{eqnarray}
(xy_{0}-yx_{0})(x\chi-y\xi)P&+&x(x_{0}\chi-y_{0}\xi)(xQ-yP),\nonumber\\
\text{ and }(xy_{0}-yx_{0})(x\chi-y\xi)Q&+&y(x_{0}\chi-y_{0}\xi)(xQ-yP),
\label{pirm}
\end{eqnarray}
as polynomials, are divisible by $(x\chi-y\xi)^{2}$. This shows, first, that $xQ-yP$ must be divisible by $(x\chi-y\xi)$:
\begin{eqnarray}
\xi Q(\xi,\chi)-\chi P(\xi,\chi)=0.\label{firr}
\end{eqnarray}
At this point we must choose which one - $\xi$ or $\chi$ - is required to be non-zero. Suppose, we require $\chi\neq 0$.
In this case, the derivative with respect to $x$ of the expression (\ref{pirm}) must also vanish at $(x,y)=(\xi,\chi)$:
\begin{eqnarray*}
\xi(x_{0}\chi-y_{0}\xi)[\xi Q_{x}(\xi,\chi)-\chi P_{x}(\xi,\chi)]=0.
\end{eqnarray*}
Thus, we demand
\begin{eqnarray}
\xi Q_{x}(\xi,\chi)-\chi P_{x}(\xi,\chi)=0.\label{secc}
\end{eqnarray}
If this is satisfied, the polynomial just above (\ref{pirm}) is also divisible by $(x\chi-y\xi)^{2}$. Note that if we were to work with the case $\chi=0$, we had to use the derivative with respect to $y$ rather than $x$.
 If the two requirements (\ref{firr}) and (\ref{secc}) are met, then $\varpi'$ and $\varrho'$ will have the denominator equal to $D$, since the factor $(x\chi-y\xi)^{2}$ will cancel out. If we can additionally achieve that
\begin{eqnarray}
 x_{0}Q(x_{0},y_{0})-y_{0}P(x_{0},y_{0})=0,\label{thirr}
\end{eqnarray}
this lowers the degree in the denominator by $1$, and we are done! Of course, $P$ and $Q$ are given as they are and generically there is no reason why the two polynomials (\ref{firr}) and (\ref{secc}) must have a common root. Also, (\ref{thirr}) requires that $xQ-yP$ must have a common root with $D$, and this might happen only by an accident. But here we are in affordance to make the following trick: let us first perform the linear change (\ref{tiesine}), given by $L^{-1}(x,y)\mapsto (\alpha x+\beta y,\gamma x+\delta y)$, $\alpha\delta-\beta\gamma\neq 0$, and then seek for an appropriate $A$. We have:
\begin{eqnarray*}
L^{-1}\circ(\varpi\bl\varrho)\circ L=\varpi_{0}(x,y)\bl\varrho_{0}(x,y)
=\frac{\alpha P(x',y')+\beta Q(x',y')}{D(x',y')}\bl \frac{\delta P(x',y')+\gamma Q(x',y')}{D(x',y')},\\
L(x,y)=(x',y')=\frac{1}{\alpha\delta-\beta\gamma}\cdot (\delta x-\beta y,-\gamma x+\alpha y).
\end{eqnarray*}
In fact, we do not need to worry about the linear change $L$: just consider two new functions $\widetilde{\varpi}(x',y')\bl\widetilde{\varrho}(x',y')=
\varpi_{0}(x,y)\bl \varrho_{0}(x,y)$. Since both $x'$ and $y'$ are linear combinations of $x$ and $y$, this will not affect the results on degrees of polynomials. With this trick in mind, we can, without loss of generality, consider
\begin{eqnarray}
\varpi(x,y)=\frac{\alpha P(x,y)+\beta Q(x,y)}{D(x,y)}\quad \varrho(x,y)=\frac{\gamma P(x,y)+\delta Q(x,y)}{D(x,y)},
\label{perein2}
\end{eqnarray}
and solve the problem of finding $A$ with an additional $4$ variables $\alpha$, $\beta$, $\gamma$, and $\delta$ in our disposition, with a single crucial restriction $\alpha\delta-\beta\gamma\neq 0$. So, we may temporarily forget about flows, think about two transformations (\ref{perein}) and (\ref{perein2}), and how they may lead to a reduction.\\

Let therefore
\begin{eqnarray*}
\alpha P(x,y)+\beta Q(x,y)&=&\widehat{P}(x,y),\\
\gamma P(x,y)+\delta Q(x,y)&=&\widehat{Q}(x,y).
\end{eqnarray*}
 Since $\deg(D)= d-2\geq 1$, let $(x_{0},y_{0})$ be one of its roots. Note that we are dealing with homogenic polynomials, so what does matter is the ratio $(x_{0}:y_{0})\in P^{1}(\mathbb{C})$. Let $(\xi:\chi)\in P^{1}(\mathbb{C})$ be an arbitrary element. We now require that the conditions (\ref{firr}), (\ref{secc}) and (\ref{thirr}) are satisfied by $\widehat{P}$ and $\widehat{Q}$ instead of $P,Q$. Thus,
\begin{eqnarray}
\left\{\begin{array}{cc}
(\alpha\chi-\gamma\xi)P(\xi,\chi)+(\beta\chi-\delta\xi)Q(\xi,\chi)&=0,\\
(\alpha\chi-\gamma\xi)P_{x}(\xi,\chi)+(\beta\chi-\delta\xi)Q_{x}(\xi,\chi)&=0,\\
(\alpha y_{0}-\gamma x_{0})P(x_{0},y_{0})+(\beta y_{0}-\delta x_{0})Q(x_{0},y_{0})&=0.
\end{array}\right.\label{lyg}
\end{eqnarray}
Or, in the matrix form: let $\mathbf{v}=(\alpha,\beta,\gamma,\delta)^{T}$, and
\begin{eqnarray*}
\mathcal{M}=\left({
\begin{array}{cccc}
\chi P(\xi,\chi)& \chi Q(\xi,\chi)&
-\xi P(\xi,\chi)&- \xi Q(\xi,\chi)  \\
\chi P_{x}(\xi,\chi)&\chi Q_{x}(\xi,\chi)&
-\xi P_{x}(\xi,\chi)&-\xi Q_{x}(\xi,\chi)\\
 y_{0} P(x_{0},y_{0})& y_{0} Q(x_{0},y_{0})&
-x_{0} P(x_{0},y_{0})&- x_{0} Q(x_{0},y_{0})
\end{array}}\right).
\end{eqnarray*}
Then $\mathcal{M}\mathbf{v}=\mathbf{0}$.  We may assume that at least one of the numbers $P(x_{0},y_{0})$ and $Q(x_{0},y_{0})$ is non-zero; otherwise the degree can be automatically reduced by factoring out the linear factor in $\varpi=PD^{-1}$ and $\varrho=QD^{-1}$. As before, we require $\chi$ to be non-zero. Now add the $\xi\chi^{-1}$ multiple of the first column of the matrix $\mathcal{M}$ to the third, and the $\xi\chi^{-1}$ multiple of the second column to the fourth to achieve the annulation of the top-right matrix $2\times 2$. Since $x_{0}\chi-y_{0}\xi$ is chosen to be non-zero, this shows that the rank of the matrix is $\geq 1$ and it is $<3$ only if
\begin{eqnarray}
P(\xi,\chi)Q_{x}(\xi,\chi)-P_{x}(\xi,\chi)Q(\xi,\chi)=0.\label{bass}
\end{eqnarray}
If the latter is not the case, the rank is $3$. If the rank is $3$, the first two equations in (\ref{lyg}) show that  $(\alpha\chi-\gamma\xi)=(\beta\chi-\delta\xi)=0$; this implies $\alpha\delta-\beta\gamma=0$, and that does not suite our needs. So, we are  forced to chose $(\xi,\chi)$ which satisfies (\ref{bass}).
At this stage we will separate three cases.\\
\noindent\fbox{{\textbf{A}.}} Assume that $P$ and $Q$ have a common root $(\xi,\chi)\neq (1,0)$, which is at least a double root for both of them. Then the first two rows of the matrix $\mathcal{M}$ vanish, and we can trivially find a vector $\mathbf{v}$ such that $\alpha\delta-\beta\gamma\neq 0$ and $\mathcal{M}\mathbf{v}=\mathbf{0}$.\\
\noindent\fbox{\textbf{B.}} Assume that for a certain pair $(\xi,\chi)$ which satisfies (\ref{bass}), at least one of $P(\xi,\chi)$ and $Q(\xi,\chi)$ is non-zero. Then the first and the third  equations in (\ref{lyg}) can be presented in an equivalent form
\begin{eqnarray}
\left({
\begin{array}{cccc}
\chi  &   &-\xi &    \\
      &\chi &   &-\xi\\
y_{0} & &-x_{0}&    \\
      & y_{0} & & -x_{0}
\end{array}}\right)\mathbf{v}=
\left({
\begin{array}{c}
\mu Q_{\star}(\xi,\chi)\\
-\mu P_{\star}(\xi,\chi)\\
\lambda Q(x_{0},y_{0})\\
-\lambda P(x_{0},y_{0})\\
\end{array}}\right),\text{ for any }\mu,\lambda\in\mathbb{C},
\label{equivv}
\end{eqnarray}
$\mathbf{v}=(\alpha,\beta,\gamma,\delta)^{T}$. (For now, ignore the subscript ``$\star$". The meaning of this will be explained later in part \textbf{C}). Thus, the solution is
\begin{eqnarray*}
\alpha(\chi x_{0}-\xi y_{0})&=&x_{0}\mu Q_{\star}(\xi,\chi)-\xi\lambda Q(x_{0},y_{0}),\\
\beta(\chi x_{0}-\xi y_{0})&=&-x_{0}\mu P_{\star}(\xi,\chi)+\xi\lambda P(x_{0},y_{0}),\\
\gamma(\chi x_{0}-\xi y_{0})&=&y_{0}\mu Q_{\star}(\xi,\chi)-\chi\lambda Q(x_{0},y_{0}),\\
\delta(\chi x_{0}-\xi y_{0})&=&-y_{0}\mu P_{\star}(\xi,\chi)+\chi\lambda P(x_{0},y_{0}),
\end{eqnarray*}
where $\mu,\lambda\in\mathbb{C}$. Now we calculate
\begin{eqnarray*}
(\alpha\delta-\beta\gamma)(\xi y_{0}-\chi x_{0})=
\mu\lambda[P_{\star}(\xi,\chi)Q(x_{0},y_{0})-P_{\star}(x_{0},y_{0})Q(\xi,\chi)].
\end{eqnarray*}
Of course we must choose $\mu,\lambda\neq 0$, but otherwise they are arbitrary. We must also demand that
\begin{eqnarray}
P_{\star}(\xi,\chi)Q(x_{0},y_{0})-P(x_{0},y_{0})Q_{\star}(\xi,\chi)\neq 0.\label{bass2}
\end{eqnarray}
Now, assume that $(x_{0},y_{0})$ is not proportional to $(1,0)$. If this is the case, we run into some difficulties, where both (\ref{bass}) and (\ref{bass2}) cannot be satisfied. For example, let us consider
\begin{eqnarray*}
P(x,y)=(xy_{0}-yx_{0})^{d}+ay^{d},\quad Q(x,y)=(xy_{0}-yx_{0})^{d}+by^d,\quad a\neq b,\quad y_{0}\neq 0.
\end{eqnarray*}
Then the only point which satisfies both (\ref{bass}) and (\ref{bass2}) is $(\xi,\chi)=(1,0)$, and this is not allowed by our previous assumption. Here we nevertheless can salvage the situation. Let us make a jump back to the beginning of the proof of our proposition and first  make a linear change and a consequent conjugation with an appropriate $A$, given by (\ref{vecconj}), to secure that the degree of denominator of $\varpi'$ and $\varrho'$ in (\ref{perein}) remains the same, but this denominator is now divisible by $y$. So, we now search for $A$ of the form
\begin{eqnarray*}
A(x,y)=\frac{xy_{0}-yx_{0}}{y},
\end{eqnarray*}
where $(xy_{0}-yx_{0})$ is a linear factor of $D$, $y_{0}\neq 0$. We want the factors $y$ and $(xy_{0}-yx_{0})$ to cancel out in the denominators of $\varpi'$ and $\varrho'$; this ensures the success of our task. Thus, we wish to have (see (\ref{perein}))
\begin{eqnarray*}
\left\{\begin{array}{r r}
x_{0}Q(x_{0},y_{0})-y_{0}P(x_{0},y_{0})=0,\\
Q(1,0)=0.
\end{array}\right.
\end{eqnarray*}
Now, perform the same trick - first use the linear conjugation. This gives only two conditions for four variables, and it is easily checked that we can always secure that $\alpha\delta-\beta\gamma\neq 0$. So, we now may return to the point (\ref{bass2}) and assume for the rest, without the loss of generality, that $(x_{0},y_{0})=(1,0)$.\\

Thus, to finish the proof of the case \textbf{B}, we will now show that the following lemma is  valid:
\begin{lem}Given two non-proportional homogenic polynomials $P$ and $Q$ of degree $d\geq 1$, such that $(x_{0},y_{0})=(1,0)$ is not a root of at least one of them. Suppose also that $P$ and $Q$ are not of the form
\begin{eqnarray}
P(x,y)=j(x,y)(x y_{1}-y x_{1}),\quad Q(x,y)=j(x,y)(x y_{2}-y x_{2}),\quad\deg(j)=d-1.
\label{except}
\end{eqnarray}
Then we can find a pair $(\xi,\chi)$, such that both (\ref{bass}) and (\ref{bass2}) are both satisfied.
\end{lem}
Note that the requirement (\ref{bass2}), if satisfied, automatically implies  $\chi\neq 0$ and that at least one of $P(\xi,\chi)$ and $Q(\xi,\chi)$ is non-zero.

Of course, if $P$ and $Q$ are proportional, the requirement (\ref{bass}) holds identically, while (\ref{bass2}) is trivially not satisfied.
\begin{proof}Let $(P,Q)=j$. Then $P=jp$, $Q=jq$, where $(p,q)=1$. By the assumption, $\deg(p)=\deg(q)\geq 1$. We have
\begin{eqnarray*}
PQ_{x}-P_{x}Q&=&jp(j_{x}q+jq_{x})-(j_{x}p+jp_{x})jq=j^{2}(pq_{x}-p_{x}q),\\
PQ_{0}-P_{0}Q&=&jj_{0}(pq_{0}-p_{0}q).
\end{eqnarray*}
Here the subscript $``0"$ stands for the special value of all the polynomials at $(x,y)=(1,0)$. So, in any case, minding the requirements of the lemma which we want a point $(x,y)=(\xi,\chi)$ to satisfy, the latter identities show that, without the loss of generality, we may assume $P$ and $Q$ to be coprime.\\

First we prove our lemma in the special case $P(1,0)=0$. Then $Q(1,0)\neq 0$. Thus, we need to find the root of $PQ_{x}-P_{x}Q$, not proportional to $(1,0)$, which is not the root of $P$; that is what (\ref{bass2}) requires. Assume, the opposite holds: for any $(\xi,\chi)$, satisfying (\ref{bass}) and $\chi\neq 0$, this implies $P(\xi,\chi)=0$. Since by our assumption of coprimality we must also have $Q(\xi,\chi)\neq 0$, this gives $P_{x}(\xi,\chi)=0$. And so, every root of $PQ_{x}-P_{x}Q$ which is not proportional to $(1,0)$ is at least a double root of $P$. Suppose, $l(x,y)=xy_{1}-yx_{1}$ is a linear polynomial, $y_{1}\neq 0$. Assume, that
\begin{eqnarray}
l^{s}||(PQ_{x}-QP_{x}),\quad s\geq 1.
\label{assump}
\end{eqnarray}
This, by the property we have just obtained, implies $l^2|P$. If $s\geq 2$, the above divisibility property gives $l^2|P_{x}$, and thus $l^{3}|P$. We can iterate this procedure. After a finite number of steps, we arrive at the conclusion that (\ref{assump}) implies
\begin{eqnarray*}
l^{s+1}||P,
\end{eqnarray*}
and the divisibility is exact. So, let
\begin{eqnarray*}
P=l_{1}^{s_{1}+1}\cdots l_{i}^{s_{i}+1}\cdot y^{\kappa};
\end{eqnarray*}
here $\kappa\geq 1$, since $P(1,0)=0$, $l_{k}=y_{k}x-x_{k}y$ are different linear factors, $S=s_{1}+\cdots +s_{i}\geq 0$, $S+i+\kappa=d$ $y_{k}\neq 0$, all $s_{k}\geq 0$.
Then
\begin{eqnarray*}
PQ_{x}-P_{x}Q=l_{1}^{s_{1}}\cdots l_{i}^{s_{i}}\cdot y^{\kappa+i+d-1},
\end{eqnarray*}
since the total degree is $2d-1$. Plugging the expression for $P$ into the above, we get
\begin{eqnarray*}
Q_{x}-Q\Bigg{(}\frac{(s_{1}+1)y_{1}}{l_{1}}+\cdots+\frac{(s_{i}+1)y_{i}}{l_{i}}\Bigg{)}=\frac{y^{i+d-1}}{l_{1}\cdots l_{i}}.
\end{eqnarray*}
This is a linear first order differential equation for $Q$. Solving it with respect to $Q$, we obtain
\begin{eqnarray*}
Q(x,y)=Cl_{1}^{s_{1}+1}\cdots l_{i}^{s_{i}+1}\cdot y^{\kappa}+
l_{1}^{s_{1}+1}\cdots l_{i}^{s_{i}+1}\int\frac{y^{i+d-1}\d x}{l_{1}^{s_{1}+2}\cdots l_{i}^{s_{i}+2}},\quad C\in\mathbb{C}.
\end{eqnarray*}
Assume that $i\geq 1$. We note that this does not imply that the last integral is divisible by  $y^{i+d-1}$. Nevertheless, decompose the fraction under the integral into simple fractions. Keep in mind that we are dealing with homogenic functions, so each summand is a homogenic rational function of degree $\kappa-1$. So, if $i\geq 1$, the generic summand looks like
\begin{eqnarray*}
\frac{y^{i+d-1}}{y^{S+2i-a}l_{j}^{a}},\quad 1\leq a\leq s_{j}+2,\quad 1\leq j\leq i.
\end{eqnarray*}
(Of course, to obtain a polynomial after integration, it is necessary that the function under integral is such that the summands with $a=1$ are not present in the simple fraction decomposition). Finally,
\begin{eqnarray*}
i+d-1-(S+2i-a)=\kappa-1+a\geq \kappa.
\end{eqnarray*}
We get that $Q$ is divisible by $y^{\kappa}$, $P$ and $Q$ are not coprime, and this shows that the assumption $i\geq 1$ leads to a contradiction. So, $i=0$, and then $Q$ is divisible by $y^{d-1}$. Thus, $d=\kappa=1$, $P=y$, and this gives that $Q$ is a linear polynomial, thus yielding exactly the exceptional case described in the formulation of the Proposition. \\
\indent Now we will turn to the general case, when $P(1,0)$ is not necessarily $0$. Let
\begin{eqnarray*}
\widetilde{P}=aP+bQ,\quad\widetilde{Q}=cP+dQ,\quad ad-bc\neq 0.
\end{eqnarray*}
Then, by a direct calculation,
\begin{eqnarray*}
\widetilde{P}\widetilde{Q}_{x}-\widetilde{P}_{x}\widetilde{Q}&=&
(ad-bc)[PQ_{x}-P_{x}Q],\\
\widetilde{P}\widetilde{Q}_{0}-\widetilde{P}_{0}\widetilde{Q}&=&
(ad-bc)[PQ_{0}-P_{0}Q].
\end{eqnarray*}
So, to prove lemma for $P,Q$ is equivalent to proving it for $\widetilde{P},\widetilde{Q}$. Now choose $a,b,c,d$ in such a way that $\widetilde{P}(1,0)=aP(1,0)+bQ(1,0)=0$, $ad-bc\neq 0$. This, of course, can be done. We therefore arrive to the case which was settled before.\end{proof}
\noindent\fbox{\textbf{C.}} Assume that $P$ and $Q$ are of the form (\ref{except}) for some $j$. Then
\begin{eqnarray*}
PQ_{x}-P_{x}Q=j^{2}y(x_{2}y_{1}-x_{1}y_{2}).
\end{eqnarray*}
Let us choose $(\xi,\chi)$ to be the root of $j$. We may assume that it is a simple root, since a multiple root case was settled before in \textbf{A}. Since now the first row of the matrix $\mathcal{M}$ vanishes identically, while the second row does not, we see that a solution of $\mathcal{M}\mathbf{v}=\mathbf{0}$ can be given in the equivalent form (\ref{equivv}), only this time the subscript ``$\star$" stands for ``$x$". Everything carries from the case \textbf{B} without alterations minding the convention on ``$\star$". We thus arrive at the condition (\ref{bass2}), which must be satisfied by $(\xi,\chi)$. In our case,
\begin{eqnarray*}
P_{x}(\xi,\chi)Q(1,0)-P(1,0)Q_{x}(\xi,\chi)=
j_{x}(\xi,\chi)j(1,0)\chi(x_{2}y_{1}-x_{1}y_{2}).
\end{eqnarray*}
This is a non-zero, and this establishes the last case. This finishes the proof of Step I. \end{proof}

\subsection{Step II }
\begin{prop}Assume that $\phi(\m{x})$ is a rational flow, and that $\varrho(x,y)\equiv 0$. Then $\varpi(x,y)=zx^2$ or $\varpi(x,y)=zy^2$ for a certain $z\in\mathbb{C}$.
\end{prop}
\begin{proof}As required, we must solve the equation (\ref{parteq}). We first need to solve (\ref{ordeq1}) in the form (\ref{ODE}). In our case, $B(w)=0$, $A(w)=\varpi(1,w)$, $y=wx$. So, the equation read as
\begin{eqnarray*}
wA(w)x'+A(w)x=1.
\end{eqnarray*}
The general solution is given by
\begin{eqnarray*}
x=\frac{C}{w}+\frac{1}{w}\int\limits_{w_{0}}^{w}\frac{\d t}{\varpi(1,t)},\quad C=y-\int\limits_{w_{0}}^{y/x}\frac{\d t}{\varpi(1,t)},\quad w_{0}\in\mathbb{C}\text{ is fixed}.
\end{eqnarray*}
The solution to the equation (\ref{ordeq2}) is $z=Dy$.  Thus, employing the standard techniques to solve linear first order PDE, we find that the  general solution to (\ref{parteq}) is given by
\begin{eqnarray}
H\Bigg{(}y-\int\limits_{w_{0}}^{y/x}\frac{\d t}{\varpi(1,t)}\Bigg{)}y,
\end{eqnarray}
for any differentiable $H$. The second boundary condition (\ref{bound}) is met by $H=H_{0}\equiv 1$, and then, as could be expected, we obtain $v(x,y)=y$. While the first boundary condition (\ref{bound}) requires that $H=H_{1}$, where the latter must satisfy
\begin{eqnarray}
H_{1}\Bigg{(}-\int\limits_{w_{0}}^{w}\frac{\d t}{\varpi(1,t)}\Bigg{)}=\frac{1}{w},\quad w\in\mathbb{C}.
\label{h0}
\end{eqnarray}
Thus, the solution $u(x,y)$ to (\ref{parteq}) and (\ref{bound}) satisfies
\begin{eqnarray}
H_{1}\Bigg{(}y-\int\limits_{w_{0}}^{y/x}\frac{\d t}{\varpi(1,t)}\Bigg{)}=\frac{u(x,y)}{y}.
\label{h1}
\end{eqnarray}
For a given pair $(x,y)$, let us define the function $r(x,y)$ by the requirement
\begin{eqnarray}
\int\limits_{y/r(x,y)}^{y/x}\frac{\d t}{\varpi(1,t)}=y.
\label{h2}
\end{eqnarray}
Substitute now this into (\ref{h1}) and use (\ref{h0}). This implies
\begin{eqnarray*}
r(x,y)=u(x,y).
\end{eqnarray*}
Thus, (\ref{h2}) describes the solution $u(x,y)$ implicitly. If needed, we can conjugate the vector field $\varpi\bl 0$ with the linear map $L(x,y)=(x+py,y)$; this does not change the second coordinate $\varrho\equiv0$. Thus, we can confine to the cases $\varpi_{1}(1,t)= 1$, $\varpi_{2}(1,t)=t^2$, $\varpi_{3}(1,t)=t$, $\varpi_{4}(1,t)=t^2+1$, and $\varpi_{5}(1,t)=t+1$. Using (\ref{h2}), after integration this gives, respectively, the flows
\begin{eqnarray*}
u_{1}(x,y)=\frac{x}{1-x},&\quad &u_{2}(x,y)=x+y^2,\\
\phi_{3}(\m{x})=xe^{y}\bl y,\quad \phi_{4}(\m{x})&=&\frac{xy+y^2\tan y}{y-x\tan y}\bl y,\quad \phi_{5}(\m{x})=\frac{xye^{y}}{x+y-xe^y}\bl y.
\end{eqnarray*} \\
\indent To start from the other end, we know that a function $u(x,y)\bl y$ is a birational plane transformation. So, when $y$ is fixed, $u(x,y)$, for almost all $y$ is a birational transformation of $P^{1}(\mathbb{C})$. So, $u$ must be a Jonqui\`{e}res transformation, and thus for certain $1-$variable rational functions $a,b,c$ and $d$, we have
\begin{eqnarray*}
\phi(\m{x})=u(x,y)\bl v(x,y)=\frac{a(y)x+b(y)}{c(y)x+d(y)}\bl y,
\end{eqnarray*}
where $ad-bc$ is not identically $0$. Without the loss of generality we may suppose that $a,b,c,d$ are polynomials, and $(a,b,c,d)=1$.
The boundary condition (\ref{bound}) gives $b(0)=0$, and that
\begin{eqnarray*}
\frac{a(0)x+b'(0)y}{d(0)}= x,
\end{eqnarray*}
provided that $d(0)$ is non-zero. This gives $b'(0)=0$ and $a(0)=d(0)\neq 0$. Calculating the vector field of $u(x,y)$ we get
\begin{eqnarray*}
\varpi(x,y)=\frac{1}{d(0)}\Big{(}[a'(0)-d'(0)]xy-c(0)x^2+\frac{b''(0)}{2}y^2\Big{)}.
\end{eqnarray*}
Thus, the vector field is in fact a quadratic form, $\varpi(1,t)$ is a quadratic polynomial, and we have already explored all possibilies. Now suppose $d(0)=0$. Then $a(0)=0$, $b(0)=b'(0)=0$. In this case, since $(a,b,c,d)=1$, we get $c(0)\neq 0$, and the condition (\ref{bound}) gives
\begin{eqnarray*}
\frac{a'(0)xy+\frac{b''(0)}{2}y^2}{c(0)x+d'(0)y}=x.
\end{eqnarray*}
This cannot hold unless $c(0)=0$ - a contradiction. \end{proof}
\begin{Note} Suppose that the condition (\ref{h2}) is satisfied by a function $r(x,y)=u(x,y)$ without branching points in $\mathbb{C}^{2}$. Then it was was shown in (\cite{alkauskas-un}, Subsection 2.4) that this neccessarily implies that $\varpi(1,t)$ is a quadratic trinomial. Thus, we get the same conclusion as above without the appeal to the property that $u(x,y)$ is a Jonqui\`{e}res transformation. The proof relies on the property that if $f(z)$ is an analytic function with a zero at $z=z_{0}$ of multiplicity $n_{0}\geq 2$, then in a small neighborhood of $z_{0}$ it attains every value exactly $n_{0}$ times, so the inverse of $f(z)$ must have a branching point at zero. This more general approach does not give any new information in the rational flow setting, but it is crucial in dealing with non-rational flows with rational vector fields \cite{alkauskas-un}.\end{Note}
\subsection{Step III }\label{sub4.3} Assume that a vector field of a flow $\phi$ is given by a pair of two quadratic forms $P\bl Q$. Let $L(x,y):(x,y)\mapsto(ax+by,cx+dy)$, $ad-bc=1$, be a non-degenerate linear transformation. Then the direct check shows that the vector field $(P\bl Q)\circ L$ is given by
\begin{eqnarray*}
P(a,c)x^2+E(a,c;b,d)xy+P(b,d)y^2\bl Q(a,c)x^2+e(a,c;b,d)xy+Q(b,d)y^2;
\end{eqnarray*}
here $E(a,c;b,d)=P(a+b,c+d)-P(a,c)-P(b,d)=aP_{x}(b,d)+cP_{y}(b,d)=bP_{x}(a,c)+dP_{y}(a,c)$ is the associated bilinear form; $e$ is defined analogously: $e(a,c;b,d)=Q(a+b,c+d)-Q(a,c)-Q(b,d)$. Let $P'\bl Q'=L^{-1}\circ(P\bl Q)\circ L$ be the vector field of a flow which is linearly conjugate to $\phi$. The coefficients of $P'$ and $Q'$ are as follows:
\begin{eqnarray}
\begin{array}{ll}

P':\left\{\begin{array}{l l}
x^2:dP(a,c)-bQ(a,c),\\
xy:dE(a,c;b,d)-be(a,c;b,d),\\
y^2:dP(b,d)-bQ(b,d),
\end{array}\right.
&
Q':\left\{\begin{array}{l l}
x^2:-cP(a,c)+aQ(a,c),\\
xy:-cE(a,c;b,d)+ae(a,c;b,d),\\
y^2:-cP(b,d)+aQ(b,d).
\end{array}\right.
\end{array}
\label{perej}
\end{eqnarray}

We will separate two cases.\\
\noindent\fbox{\textbf{A.}} Assume that the polynomial $yP(x,y)-xQ(x,y)$ is not a cube of a linear polynomial. Then if we choose $(b:d)$ and $(a:c)$ to be its two different roots, then $ad-bc\neq 0$, and thus we get the following: a pair $P\bl Q$ after a conjugation with a linear change can be transformed into the pair $P'\bl Q'=\varpi\bl\varrho$, where $\varpi(x,y)=ax^2+bxy$, $\varrho(x,y)=cxy+dy^2$. The case when one of $b$ or $c$ is $0$ will be settled in the Step IV. And so we assume $b,c\neq 0$. Let the solution to (\ref{parteq}) with the first boundary condition (\ref{bound}) be given by the formal series
\begin{eqnarray*}
u(x,y)=\sum\limits_{n=0}^{\infty}x^{n}f_{n}(y),\quad f_{n}(y)=\frac{1}{n!}\frac{\partial^{n}}{\partial x^{n}}u(x,y)\Big{|}_{x=0}\in\mathbb{C}(y).
\end{eqnarray*}
Negative powers of $x$ are not present because of (\ref{seru}) and (\ref{upow}). These identities also show that $f_{n}(y)$, if expanded into powers of $y$, contain only non-negative powers. The boundary condition (\ref{bound}) is then equivalent to $y^2|f_{0}(y)$ and $f_{1}(0)=1$. Thus,
\begin{eqnarray*}
u_{x}(x,y)=\sum\limits_{n=0}^{\infty}nx^{n-1}f_{n}(y),\quad u_{y}(x,y)=\sum\limits_{n=0}^{\infty}x^{n}f'_{n}(y).
\end{eqnarray*}
Plugging this into (\ref{parteq}) gives
\begin{eqnarray}
\sum\limits_{n=0}^{\infty}nx^{n-1}f_{n}(y)(ax^2+bxy-x)+\sum\limits_{n=0}^{\infty}x^{n}f'_{n}(y)(cxy+dy^2-y)=-\sum\limits_{n=0}^{\infty}x^{n}f_{n}(y).
\label{expan}
\end{eqnarray}
Comparing the coefficients at $x^{0}$, we obtain
\begin{eqnarray*}
f'_{0}(y)(dy^2-y)=-f_{0}(y).
\end{eqnarray*}
Integrating this and using the fact $y^2|f_{0}(y)$, we get $f_{0}(y)\equiv 0$. The coefficient of (\ref{expan}) at $x^{1}$ gives
\begin{eqnarray*}
f_{1}(y)(by-1)+f'_{1}(y)(dy^2-y)=-f_{1}(y).
\end{eqnarray*}
This linear ODE has a rational solution only if $d\neq 0$, so we henceforth assume this. Solving this with the condition $f_{1}(0)=1$ yields the solution
\begin{eqnarray*}
f_{1}(y)=(1-dy)^{-\frac{b}{d}}.
\end{eqnarray*}
Since this must be a rational function, we obtain $\frac{b}{d}=B\in\mathbb{Z}$. Analogously, solving (\ref{parteq}) for $v(x,y)$ with the second boundary condition (\ref{bound}) will yield the necessary condition $a\neq 0$ and $\frac{c}{a}=C\in\mathbb{Z}$. Thus, after an additional conjugation of the vector field $\varpi\bl\varrho$ with $(x,y)\mapsto (x/a,y/d)$, we may assume, without the loss of generality, $a=d=1$, and $P=\varpi=x^2+Bxy$, $Q=\varrho=Cxy+y^2$, $B,C\in\mathbb{Z}$. In general setting of $B,C$,
compare the coefficient of (\ref{expan}) at $x^2$. This gives
\begin{eqnarray}
f_{2}(y)(2By-1)+f'_{2}(y)(y^2-y)=-f_{1}(y)-f'_{1}(y)Cy.
\label{fun2}
\end{eqnarray}
In this step we do not write explicitly the solutions but rather rely on symbolic calculations executed with the help of MAPLE. After having examined  the explicit expression for the solution, we must distinguish three cases. In case $B=1$ there exists a rational solution to the above differential equation only if $C=1$. But then $x\varrho-y\varpi\equiv 0$, and we obtain the level $0$ flow; see the end of the Subection \ref{sub3.2}. In case $B=2$ the solution involves $C\log(1/(1-y))$ and no solution of this differential equation is rational. Assume that $B\neq 1,2$. Analogous analysis of the solution $v(x,y)$ shows that we may assume $C\neq 1,2$. Consider (\ref{perej}). In our case, $yP(x,y)-xQ(x,y)=xy\big{(}(B-1)y-(C-1)x\big{)}$. So, let us choose
\begin{eqnarray*}
\left(
  \begin{array}{cc}
    a & b \\
    c & d \\
  \end{array}\right)=
  \left(
    \begin{array}{cc}
      \frac{1}{C-1} & B-1 \\
      0 & C-1 \\
    \end{array}
  \right).
\end{eqnarray*}
Then (\ref{perej}) gives
\begin{eqnarray*}
P'(x,y)=\frac{1}{C-1}\,x^2+(B+C-2)xy,\quad
Q'(x,y)=\frac{C}{C-1}\,xy+(BC-1)y^2.
\end{eqnarray*}
This is a vector field of a rational flow, and so we know that $BC-1\neq 0$. Conjugating this with the linear map $(x,y)\mapsto\big{(}(C-1)x,\frac{y}{BC-1}\big{)}$, we obtain
\begin{eqnarray}
P''(x,y)=x^2+\frac{B+C-2}{BC-1}\, xy,\quad
Q''(x,y)=Cxy+y^2.\label{auto}
\end{eqnarray}
This is also a vector field of a rational flow. This implies the arithmetic condition
\begin{eqnarray}
\frac{B+C-2}{BC-1}\in\mathbb{Z}.\label{integer}
\end{eqnarray}
Assume $B+C=2$. Let $A(x,y)=-\frac{y}{x+y}$. Then by a direct calculation, using (\ref{vecconj}), we obtain
\begin{eqnarray*}
\varpi'(x,y)=(C-2)xy,\quad\varrho'(x,y)=-y^2,
\end{eqnarray*}
and thus the flow with the vector field $\varpi\bl\varrho$ is $\ell-$conjugate to the canonical flow $\phi_{C-1}$. If $B+C\neq 2$, and $B,C\neq 0,1$ (conditions $B,C\neq 2$ are not needed), we are left with $10$ pairs of integers $(B,C)$ which satisfy (\ref{integer}):
\begin{eqnarray*}
(-2,-1)\leftrightarrow(-5,-1),\quad(-3,-1)\circlearrowleft,\\
(-1,-2)\leftrightarrow(-5,-2),\quad (-2,-2)\circlearrowleft,\\
(-1,-3)\leftrightarrow(-3,-3),\quad (-2,-5)\leftrightarrow(-1,-5).
\end{eqnarray*}
The symbol $``\leftrightarrow"$ means that the two pairs a $l-$conjugate via (\ref{auto}), and $``\circlearrowleft"$ means that the flow is self $l-$conjugate. The pairs $(B,C)$ and $(C,B)$ are also $l-$conjugate with the help of the involution $i_{0}(x,y)=(y,x)$. Choosing one pair from each equivalence class, we are left to explore three cases
\begin{eqnarray*}
(B,C)=(-2,-2),\quad (B,C)=(-3,-3),\quad (B,C)=(-1,-2).
\end{eqnarray*}
These can be proved to arise from non-rational flows by calculating that the orbits of all of them are curves of genus $1$ and thus cannot be parametrized by rational functions. For example, consider $\varpi(x,y)=x^2-2xy$, $\varrho(x,y)=-2xy+y^2$. This is a very remarkable vector field: it is invariant under conjugation with two independent linear involutions $(x,y)\mapsto(y,x)$ and $(x,y)\mapsto(x-y,-y)$. So, the flow $\Lambda(x,y)=u(x,y)\bl v(x,y)$ is invariant under these two involutions, too. The invariance under the first involution gives the condition $u(x,y)=v(y,x)=\lambda(x,y)$, and the second involution yields the identities
\begin{eqnarray*}
\left\{\begin{array}{c@{\qquad}l}
\lambda(x,y)+\lambda(-y,x-y)+\lambda(y-x,-x)=0,\\
\lambda(x,y)+\lambda(-x,y-x)=0.
\end{array}\right.
\end{eqnarray*}

 The functions $f_{n}(y)$ are not just rational functions but in fact all are polynomials:
\begin{eqnarray*}
f_{1}(y)&=&1-2y+y^2,\\
f_{2}(y)&=&1-y,\\ 
f_{3}(y)&=&1-2y+\frac{5}{2}y^2-2y^3+y^4-\frac{2}{7}y^5+\frac{1}{28}y^6,
\\ f_{4}(y)&=&1-\frac{5}{2}y+3y^2-2y^3+\frac{5}{7}y^4-\frac{3}{28}y^5,\\
f_{5}(y)&=&1-3y+\frac{9}{2}y^2-\frac{32}{7}y^3+\frac{51}{14}y^4-\frac{33}{14}y^5+\frac{33}{28}y^6-\frac{3}{7}y^7+\frac{3}{28}y^8
-\frac{3}{182}y^9+\frac{3}{2548}y^{10},
\end{eqnarray*}
and so on. By the direct calculation, using (\ref{upow}), we get
\begin{eqnarray*}
\lambda(x,x)=\frac{x}{1+x},\quad \lambda(x,0)=\frac{x}{1-x},\quad \lambda(0,x)=0,
\end{eqnarray*}
but
\begin{eqnarray*}
\lambda(x,-x)=x+3x^2+3x^3+3x^4+6x^5+9x^6+12x^7+\frac{117}{7}x^8+\frac{171}{7}x^9+\frac{246}{7}x^{10}
+\frac{348}{7}x^{11}\\
+\frac{495}{7}x^{12}+\frac{708}{7}x^{13}+\frac{13140}{91}x^{14}+\frac{131076}{637}x^{15}+\frac{186903}{637}x^{16}
+\frac{266670}{637}x^{17}
+\frac{380403}{637}x^{18}\\
+\frac{542532}{637}x^{19}+\frac{1130958}{931}x^{20}+\frac{20971530}{12103}x^{21}+\frac{209391300}{84721}x^{22}
+\frac{29866154}{84721}x^{23}+\cdots.
\end{eqnarray*}
Calculations show that prime numbers appearing in the denominators are unbounded: for example, the denominator of the $100$th coefficient is
\begin{eqnarray*}
5^{2}\cdot7^{16}\cdot13^{7}\cdot19^{5}\cdot31^{3}\cdot37^{2}\cdot43^{2}\cdot61\cdot67\cdot73\cdot79\cdot97,
\end{eqnarray*}
and the largest prime factor of the denominator of the $200$th coefficient is $199$; thus, $\lambda(x,-x)$, and therefore $\lambda(x,y)$ cannot be a rational function; this is an empiric reason. Rigorously: the differential equation (\ref{orbits}) in this case possesses a rational solution for $N=3$, and it is given by
\begin{eqnarray*}
\mathscr{W}(x,y)=xy(x-y).
\end{eqnarray*}
Thus,
\begin{eqnarray*}
\frac{\lambda(xz,yz)}{z}\cdot
\frac{\lambda(yz,xz)}{z}\cdot\Big{(}\frac{\lambda(xz,yz)}{z}-
\frac{\lambda(yz,xz)}{z}\Big{)}\equiv xy(x-y),\quad x,y,z\in\mathbb{R}.
\end{eqnarray*}
Consider the projective cubic $XY(X-Y)=Z^3$. All partial derivatives vanish only at $(X,Y,Z)=(0,0,0)$, su this cubic is non-singular, it is thus an elliptic curve, and therefore $\lambda(xz,yz)/z$, $\lambda(yz,xz)/z$ cannot both be rational functions. 
Other two cases $(B,C)=(-3,-3)$ and $(B,C)=(-1,-2)$ can also be proved to arise from non-rational flows by similar analysis. These turn out to be level $4$ and $6$ flows, respectively, with the equations for orbits given by
\begin{eqnarray*}
\mathscr{W}(x,y)=xy(x-y)^2\text{ and }
\mathscr{W}(x,y)=(3x-2y)x^{3}y^2.
\end{eqnarray*}
These curves are also birationaly equivalent to elliptic curves.
\begin{Note}\label{note5} The main body of the paper \cite{alkauskas-un} is about these three exceptional vector fields. For example, it turns out that $\lambda(x,-x)/x$ is an elliptic function with the full period lattice given by $\mathbb{Z}(\pi_{3}2^{-1/3})\oplus\mathbb{Z}(\pi_{3}2^{-1/3}e^{2\pi i/3})$; here the constant $\pi_{3}$ is given by
\begin{eqnarray*}
\pi_{3}=\frac{\sqrt{3}}{2\pi}\Gamma\Big{(}\frac{1}{3}\Big{)}^{3}=5.299916_{+}.
\end{eqnarray*}
In general, for fixed $x,y$, $xy(x-y)\neq 0$, the function $\lambda(xz,yz)/z$ is an elliptic function in variable $z\in\mathbb{C}$ with hexagonal period lattice. The function $\lambda(x,y)$ can be given an analytic expression in terms of \emph{Dixonian elliptic functions} ${\rm sm}(z)$ and ${\rm cm}(z)$, see \cite{flajolet-c}. Thus, if we denote $\varsigma=[xy(x-y)]^{1/3}$, $A={\rm sm}(\varsigma)\varsigma^{-1}$, $B={\rm cm}(\varsigma)$, then
\begin{eqnarray*}
\lambda(x,y)=\frac{x(x-y)(B-AB^{2}y+A^2xy)^{2}}
{(x-B^3y)(B^2-Ax+A^2Bxy)}.
\end{eqnarray*}
Similar analytic expressions exist for other two vector fields. These questions are treated in \cite{alkauskas-un}. Also, it turns out that the PDE system (\ref{parteq}), even in the setting of the current paper, is governed by the following system of ODE:
\begin{eqnarray*}
\left\{\begin{array}{l}
a'(x)=-\varpi(a(x),b(x)),\\
b'(x)=-\varrho(a(x),b(x)),\\
\mathscr{W}(a(x),b(x))=1.
\end{array}
\right.
\end{eqnarray*}
This question will be discussed in \cite{alkauskas2}. We finally note that it turns that the close relative of a system (\ref{parteq}) has already appeared in the literature in relation with P\'{o}lya urn models, mentioned in the Note \ref{note3}. See, for example, (\cite{flajolet}, p. 1206, formula (14)).
\end{Note}
\noindent\fbox{\textbf{B.}} We are left to consider the case when $yP(x,y)-xQ(x,y)$ is a cube of a linear polynomial. After performing the linear conjugation, we may assume that this linear polynomial is $x^3$. Let $P(x,y)=ax^2+bxy+cy^2$, $Q(x,y)=\alpha x^2+\beta xy+\gamma y^2$. Since $yP-xQ=x^3$, we get
\begin{eqnarray*}
P(x,y)=ax^2+bxy,\quad Q(x,y)=-x^2+axy+by^2.
\end{eqnarray*}
The case $b=0$ will be settled in the step IV. Assume that $b\neq 0$. If $a\neq 0$, after the conjugation with the linear map $(x,y)\mapsto (x/a,y/b)$ we get the vector field
\begin{eqnarray*}
x^2+xy\bl \lambda x^2+xy+y^2,\quad\lambda\neq 0.
\end{eqnarray*}
Using the same method as in the part \textbf{A}, we will show that in this case the solution to (\ref{parteq}) with the first condition (\ref{bound}) is not rational. Indeed, if again
\begin{eqnarray*}
u(x,y)=\sum\limits_{n=0}^{\infty}x^{n}f_{n}(y)
\end{eqnarray*}
is the solution, then
\begin{eqnarray*}
\sum\limits_{n=0}^{\infty}nx^{n-1}f_{n}(y)(x^2+xy-x)+\sum\limits_{n=0}^{\infty}x^{n}f'_{n}(y)(\lambda x^2+xy+y^2-y)=-\sum\limits_{n=0}^{\infty}x^{n}f_{n}(y).
\end{eqnarray*}
In the same manner, we get $f_{0}(y)=0$, $f_{1}(y)=(1-y)^{-1}$, $f_{2}(y)=(1-y)^{-2}$. Further, comparing the coefficient at $x^3$ of the above yields the equation for $f_{3}(y)$:
\begin{eqnarray*}
f_{3}(y)(3y-2)+f'_{3}(y)(y^2-y)=-2f_{2}(y)-f'_{2}(y)y-f'_{1}(y)\lambda.
\end{eqnarray*}
The solution is given $-\frac{\lambda\log(1-y)}{y^2(y-1)}+$rational function, and thus this ODE has no rational solution for $\lambda\neq 0$. All our calculations can be double-checked using the recurrence (\ref{upow}), which is also easy to implement in MAPLE and provides an alternative way to calculate the Taylor coefficients of $f_{n}(y)$.\\
\indent If $a=0$, after the conjugation with the linear map $(x,y)\mapsto (x,y/b)$, we get the vector field
\begin{eqnarray*}
xy\bl \lambda x^2+y^2,\quad\lambda\neq 0.
\end{eqnarray*}
We treat this exactly the same way as the previous one. Here we find that $f_{0}(y)=0$, $f_{1}(y)=\frac{1}{1-y}$, $f_{2}(y)=0$, and that $f_{3}(y)$ is not a rational function: its expression also involves $-\frac{\lambda\log(1-y)}{y^2(y-1)}$.
\subsection{Step IV}
\begin{prop}Suppose that $\varpi(x,y)=Ux^2+Vxy+Wy^2$, $\varrho(x,y)=-y^2$, and that the first equation of (\ref{parteq}) with the first boundary condition (\ref{bound}) has a rational solution. Then, for a certain positive integer $N$, $(V+1)^2-4UW=N^2$, and the solution is then indeed rational.
\end{prop}
\begin{proof}Suppose $(V+1)^2-4UW=\Delta^{2}$, $\Delta\in\mathbb{C}$. The equation (\ref{ordeq2}) read as
\begin{eqnarray*}
\frac{\d z}{\d y}=\frac{z}{y^2+y},
\end{eqnarray*}
and thus $z=cy(y+1)^{-1}$, $c\in\mathbb{C}$. Let the general solution of (\ref{ordeq1}), or, considering it in the form (\ref{ODE}), be given by $x=f_{1}(w)+C f_{0}(w)$, $C\in\mathbb{C}$. The theory of linear first order PDEs implies that the general solution to (\ref{parteq}) is given by
\begin{eqnarray}
H\Bigg{(}\frac{x-f_{1}(y/x)}{f_{0}(y/x)}\Bigg{)}\cdot\frac{y}{y+1},
\label{soll}
\end{eqnarray}
for any differentiable $H$. Thus, the first condition of (\ref{bound}) is met by the function $H$, for which
\begin{eqnarray}
H\Bigg{(}-\frac{f_{1}(y/x)}{f_{0}(y/x)}\Bigg{)}\equiv\frac{x}{y}.
\label{condit}
\end{eqnarray}
So, we now solve (\ref{ODE}). In our case, this equation reads as
\begin{eqnarray}
x'(Ww^3+(V+1)w^2+Uw)+x(Ww^2+Vw+U)=1.\label{specODE}
\end{eqnarray}
\noindent\fbox{\textbf{A.}} First, assume $W\neq 0$. Let $Ww^2+(V+1)w+U=W(w-w_{0})(w-w_{1})$. Choose the sign for $\Delta=\sqrt{\Delta^{2}}$ and the order of $(w_{0},w_{1})$ in such a way that $w_{0}-w_{1}=\frac{\Delta}{W}$. We then have the decomposition
\begin{eqnarray*}
\frac{1}{Ww^2+(V+1)w+U}=
\frac{1}{\Delta(w-w_{0})}-\frac{1}{\Delta (w-w_{1})}.
\end{eqnarray*}

Using this, standard methods to solve linear first order ODEs give the solution to (\ref{specODE}):
\begin{eqnarray*}
x=-\frac{1}{w}+\frac{C}{w}\Big{(}\frac{w-w_{0}}{w-w_{1}}\Big{)}^{\frac{1}{\Delta}}.
\end{eqnarray*}
Let us return to PDE (\ref{parteq}). The condition (\ref{condit}) requires that
\begin{eqnarray*}
H\Bigg{(}\Big{(}\frac{w-w_{1}}{w-w_{0}}\Big{)}^{\frac{1}{\Delta}}\Bigg{)}=\frac{1}{w}.
\end{eqnarray*}
Thus,
\begin{eqnarray*}
H(\alpha)=\frac{\alpha^{\Delta}-1}{\alpha^{\Delta}w_{0}-w_{1}}.
\end{eqnarray*}
This, via (\ref{soll}), this gives the solution of (\ref{parteq}) with the first boundary condition (\ref{bound}):
\begin{eqnarray*}
\frac{(y+1)^{\Delta}(y-w_{1}x)-(y-w_{0}x)}
{(y+1)^{\Delta}(y-w_{1}x)w_{0}-(y-w_{0}x)w_{1}}\cdot\frac{y}{y+1}.
\end{eqnarray*}
This is a rational function only if $\Delta$ is an integer. If it is zero, then $w_{0}=w_{1}$, and the reasoning is invalid. In this case there does not exist a rational solution; in the Subsection \ref{sub5.2} we will talk a bit more about this. Changing the sign of $\Delta$ leaves the solution intact, as it only interchanges $w_{0}$ and $w_{1}$. The solution we have obtained is in fact equal to $\mathcal{W}^{(N)}_{\sigma,\tau}$ for (see (\ref{uniN}))
\begin{eqnarray*}
\Delta=N,\quad w_{0}=\tau,\quad w_{1}=\tau-\frac{N}{\sigma},\quad W=\sigma.
\end{eqnarray*}
\noindent\fbox{\textbf{B.}} Suppose now that $W=0$. Then $\Delta=V+1$. Similarly as in case $W\neq 0$, we get that the solution to (\ref{specODE}) is given by
\begin{eqnarray*}
x=-\frac{1}{w}+\frac{C}{w}\Big{(}\Delta w+U\Big{)}^{\frac{1}{\Delta}}.
\end{eqnarray*}
The condition (\ref{condit}) requires that
\begin{eqnarray*}
H\Big{(}(\Delta w+U)^{-\frac{1}{\Delta}}\Big{)}=\frac{1}{w}.
\end{eqnarray*}
Thus,
\begin{eqnarray*}
H(\alpha)=\frac{\Delta}{\alpha^{-\Delta}-U}.
\end{eqnarray*}
This, via (\ref{soll}), gives the solution to the first equation of (\ref{parteq}):
\begin{eqnarray*}
u(x,y)=\frac{\Delta x}{(y+1)^{-\Delta}(\Delta y+Ux)-Ux}\cdot\frac{y}{y+1}.
\end{eqnarray*}
This is a rational function only if $\Delta$ is an integer. If it is positive and equal to $N$, we get exactly the flow $\mathcal{W}^{(N)}_{\sigma,\tau}$ for
\begin{eqnarray*}
\Delta=N,\quad \sigma=0,\quad \tau=-\frac{U}{N}.
\end{eqnarray*}
 If $\Delta=-N$ is negative, we get exactly the solution $\mathcal{W}_{\kappa}^{(N)}$ given by (\ref{kapa}), where
\begin{eqnarray*}
N=-\Delta,\quad \kappa=\frac{U}{N}.
\end{eqnarray*}
\end{proof}
Steps V and VI were explained before, and this finishes the proof of the main Theorem.
\subsection{The maps $p$ and $\hat{p}$}\label{sub4.5}Now we proceed with the proof of Propositions \ref{propuni} and \ref{mthm2}.
\begin{proof}As we have seen, for any flow $\phi$ there exists a $1-$BIR $\ell$ such that $\ell^{-1}\circ\phi\circ\ell=\phi_{N}$. Such $\ell$ must not be of the form (\ref{birat}). However, Proposition \ref{prop16} and the main Theorem imply that for each rational flow $\phi$ there does exist an $\ell$ of the form (\ref{birat}) such that the vector field $\varpi\bl\varrho$ of $\ell^{-1}\circ\phi\circ\ell$ is $l-$conjugate to the vector field of $\phi_{N}$. That is, according to (\ref{tiesine}), for a certain linear map $L(x,y)=(ax+by,cx+dy)$, $ad-bc\neq 0$, it is of the form
\begin{eqnarray*}
\varpi\bl\varrho=L^{-1}\circ[(N-1)xy&\bl&(-y^2)]\circ L\\
=(cx+dy)\Big{(}(Nad-1)x+Nbdy\Big{)}&\bl&-(cx+dy)\Big{(}Nacx+(Nbc+1)y\Big{)};
\end{eqnarray*}
We may assume, without the loss of generality, that $ad-bc=1$, since the scaling $(x,y)\mapsto(zx,zy)$ can be included into the $1-$BIR $\ell$. Let us choose
\begin{eqnarray*}
A(x,y)=\frac{y}{cx+dy},
\end{eqnarray*}
and calculate the vector field (\ref{vecconj}). We obtain
\begin{eqnarray*}
\varpi'\bl\varrho'=Nacx^2+[N(bc+ad)-1]xy+Nbdy^2\bl (-y^2)\\
=Ux^2+Vxy+Wy^2\bl(-y^2),\quad (V+1)^2-4UW=N^2.
\end{eqnarray*}
This shows that any flow can be transformed by $\ell$ of the form (\ref{birat}) into the univariate flow.
We are left to show that such representation for level $N\geq 2$ flows is unique, while level $1$ flows can be given even simpler univariate representation, as formulated in the second part of the Proposition \ref{mthm2}. To show this, we must find all $\ell$ of the form (\ref{birat}) which preserve the univariate form of the flow, and thus we need to solve the differential equation (\ref{differ}) for $\varpi(x)=Ux^2+Vx+W$, $\varrho(x)=-1$.\\
\noindent\fbox{\textbf{A.}}  Assume that $U\neq 0$. Then we can substitute $W=((V+1)^2-N^2)(4U)^{-1}$, and the solution of (\ref{differ}) is given by
\begin{eqnarray*}
1+c\sqrt[N]{\frac{2Ux+V+1+N}{2Ux+V+1-N}},\quad c\in\mathbb{R}.
\end{eqnarray*}
The rational solution is indeed unique for $N\geq 2$. For $N=1$, there exists a family of rational solutions. Let therefore $\varpi=Ux^2+Vxy+Wy^2$, $\varrho=-y^2$, $(V+1)^2-4UW=1$, and
\begin{eqnarray*}
A(x,y)=\frac{2Uxc+(Vc+2c-2)y}{2Ux+Vy},\quad c\in\mathbb{R}.
\end{eqnarray*}
We calculate now the vector field (\ref{vecconj}). The second coordinate, as desired, is $(-y^2)$, while the first coordinate is given by
\begin{eqnarray*}
cUx^2+(Vc+2c-2)xy+\frac{(Vc+2c-2)(V+2)}{4U}y^2.
\end{eqnarray*}
Assume that $V\neq -2$. Then we choose $c$ to be $\frac{2}{V+2}$, and this gives the vector field of the form $\frac{2U}{V+2} x^{2}$. This corresponds to the univariate flow of the form $\frac{x}{\tau x+1}$. Assume that $V=-2$. Then we choose $c$ to be $0$, and this gives the vector field of the form $-2xy$. This corresponds to the univariate flow $\frac{x}{(y+1)^2}$.\\
\noindent\fbox{\textbf{B.}} Now return to the general case of level $N$. Assume that $U=0$. Then for level $N\geq 2$ we get the same conclusion, while for level $N=1$ we must have $V+1=\pm 1$, and so the vector field is either $\varpi(x,y)=Wy^2$ or it is $\varpi(x,y)=-2xy+Wy^2$. In the latter case, we also solve the differential equation (\ref{differ}). So, let
\begin{eqnarray*}
A(x,y)=\frac{xc+(1-Wc)y}{y},\quad c\in\mathbb{R}.
\end{eqnarray*}
Then the vector field (\ref{vecconj}) is
\begin{eqnarray*}
-cx^2+(2Wc-2)xy+(W-W^2c)y^2.
\end{eqnarray*}
If $W=0$, we put $c=0$ to get $-2xy$. If $W\neq 0$, we put $c=\frac{1}{W}$ to get the vector field $-\frac{1}{W}x^2$.
If we similarly treat the case $\varpi(x,y)=Wy^2$, we obtain that it can be transformed into $0$, and this corresponds to $\tau x^2$ for $\tau=0$.\end{proof}
\subsection{Reconstruction of the level from the vector field}\label{sub4.6}We will finish this section with the proof of Proposition \ref{ppp}.
\begin{proof}Given the flow $\phi$, there are several procedures how to calculate its level. Let $\varpi'(x,y)=Ux^2+Vxy+Wy^2$, $\varrho'(x,y)=-y^2$. We know that all vectors fields of flows of level $N$ are given by (\ref{vecconj}), where $A$ is any $0-$homogenic rational function, and $(V+1)^{2}-4UW=N^2$ (only in (\ref{vecconj}), for convenience reasons, we interchanged the roles of $\varpi'\bl\varrho'$ and $\varpi\bl\varrho$). From the other end, given  $\varpi(x,y)$ and $\varrho(x,y)$, we want to find $A$.
Let $Q(x,y)=Ux^2+(V+1)xy+Wy^2$. In other words, we want to solve the equation in $A(x,y)$, given by (\ref{vecconj}):
\begin{eqnarray*}
\left(\begin{array}{cc} Q-xy & -xQ \\-y^2 & -yQ \\  \end{array}\right)
\cdot\left( \begin{array}{c}  A(x,y) \\A_{x}(x,y) \\  \end{array}\right)=
\left(\begin{array}{c}\varpi(x,y) \\ \varrho(x,y) \\ \end{array}\right).
\end{eqnarray*}
If $y\varpi(x,y)-x\varrho(x,y)\equiv 0$, the corollary after (\ref{vecconj}) shows that $\varpi'(x,y)=-xy$, $Q=0$, we have a level $0$ flow, and the only solution is given by $A(x,y)=-\varpi(x,y)(xy)^{-1}$. Assume that $y\varpi(x,y)-x\varrho(x,y)$ is not identically $0$. Then $N\neq 0$, $Q\neq 0$. The determinant of the $2\times 2$ matrix is $-yQ^2$. So, the solution is given by
\begin{eqnarray*}
\left( \begin{array}{c}  A(x,y) \\A_{x}(x,y) \\  \end{array}\right)=
-\frac{1}{yQ^2}\left(\begin{array}{cc} -yQ & xQ \\y^2 & Q-xy \\  \end{array}\right)\cdot
\left(\begin{array}{c}\varpi(x,y) \\ \varrho(x,y)\\ \end{array}\right).
\end{eqnarray*}
Thus,
\begin{eqnarray*}
A(x,y)&=&\frac{y\varpi(x,y)-x\varrho(x,y)}{yQ},\\
A_{x}(x,y)&=&\frac{x\varrho(x,y)-y\varpi(x,y)}{Q^2}-\frac{\varrho(x,y)}{yQ}.
\end{eqnarray*}
Now, take the derivative of the first identity with respect to $x$. We should obtain the second identity. And conversely - if the derivative of the first is equal to the second, then we define the function $A$ by the first equality, and the vector field $\varpi\bl\varrho$ is obtained from $\varpi'\bl\varrho'$ via a conjugation with $xA(x,y)\bl yA(x,y)$. So, after some calculations, we obtain the necessary and sufficient condition
\begin{eqnarray*}
\frac{\varpi(x,y)y-\varrho(x,y)x}{\varpi_{x}(x,y)y-\varrho_{x}(x,y)x}=\frac{Q}{Q_{x}-y}=\frac{Ux^2+(V+1)xy+Wy^2}{2Ux+Vy}.
\label{sal}
\end{eqnarray*}
For a given vector field $\varpi\bl \varrho$, let $L$ denote the left hand side of this identity. Then
\begin{eqnarray*}
\frac{2L-x}{y}=\frac{(V+2)x+2Wy}{2Ux+Vy}.
\end{eqnarray*}
Direct calculation, using the property $x\varpi_{x}+y\varpi_{y}=2\varpi$ (and the same for $\varrho$) gives exactly the Proposition \ref{ppp}. \end{proof}
\section{More properties of flows; final remarks}\label{sec5}
\subsection{The orbits of a point and the level of a flow}\label{sub5.1}Here we will prove the Corollary \ref{corr1}. The orbit of a particular point $\m{x}$ under the action of the canonical flow $\phi_{N}$ is given by
\begin{eqnarray*}
u\bl v=\frac{\phi_{N}(\m{x}z)}{z}=x(yz+1)^{N-1}\bl\frac{y}{yz+1}.
\end{eqnarray*}
Thus, independently of $z$,
\begin{eqnarray*}
\mathscr{W}_{N}(\phi_{N};u,v)=uv^{N-1}=xy^{N-1},
\end{eqnarray*}
and the orbit of a particular point $(x,y)$ is given by the $N$th degree algebraic curve $\mathscr{W}_{N}(u,v)=\mathscr{W}_{N}(x,y)={\rm const}$. This, via the main Theorem, shows that in general, the orbits of a point under a flow $\phi$ of level $N$ can be given by the equation $\mathscr{W}(u,v)=\mathscr{W}(x,y)$, where $\mathscr{W}$ is homogenic $N$th degree rational function. To prove this, we only have to trace down how the orbits change under linear conjugation and conjugation with $\ell$ of the form (\ref{birat}).
\begin{prop}Let $\phi$ be a flow of level $N$. Then we have
\begin{eqnarray*}
\mathscr{W}(\ell^{-1}_{P,Q}\circ\phi\circ\ell_{P,Q};u,v)&=&\frac{P^{N}(u,v)}{Q^{N}(u,v)}\cdot\mathscr{W}(\phi;u,v);  \\
\mathscr{W}(L^{-1}\circ\phi\circ L;u,v)&=&\mathscr{W}\circ L(u,v);
\end{eqnarray*}
here $L$ is a non-degenerate linear transformation.
\end{prop}
\begin{proof}Let $\phi(x,y)=u(x,y)\bl v(x,y)$. Assume that the orbits of this flow are given by $\mathscr{W}(\phi;u,v)={\rm const.}$ The very definition of $\ell_{P,Q}$ implies that
\begin{eqnarray*}
 \ell^{-1}_{P,Q}\circ\phi\circ\ell_{P,Q}(x,y)=\frac{u'Q(u',v')}{P(u',v')}\bl \frac{v'Q(u',v')}{P(u',v')}\mathop{=}^{{\rm def}}u\bl v,\text{ where}\\
u'=u(\ell_{P,Q}(x,y)),\quad v'=v(\ell_{P,Q}(x,y)),\quad\mathscr{W}(\phi;u',v')={\rm const.}
\end{eqnarray*}
Since $u'=\frac{uP(u,v)}{Q(u,v)}$, $v'=\frac{vP(u,v)}{Q(u,v)}$, we have
\begin{eqnarray*}
 \mathscr{W}(\ell^{-1}_{P,Q}\circ\phi\circ\ell_{P,Q};u,v)=\mathscr{W}(\phi,u',v')=\frac{P^{N}(u,v)}{Q^{N}(u,v)}\cdot\mathscr{W}(\phi;u,v).
\end{eqnarray*}
The second identity is checked directly. \end{proof}
Looking at the structure of level $0$ flows (\ref{lel0}), we see that these all satisfy $\mathscr{W}(u,v)=uv^{-1}$. The second part of the above proposition claims that, however, the orbits should be given by
\begin{eqnarray*}
\frac{au+bv}{cu+dv}={\rm const.},\quad ad-bc\neq 0.
\end{eqnarray*}
But this, of course, implies $uv^{-1}={\rm const.}$ (for another constant). Practically, for any $N$, if we have already calculated the vector field, it is simpler to find the function $\mathscr{W}$ directly not relying on the above proposition. Indeed, differentiating the identity
\begin{eqnarray*}
\mathscr{W}\Big{(}\frac{u(xz,yz)}{z},\frac{v(xz,yz)}{z}\Big{)}\equiv\textrm{const}.
\end{eqnarray*}
with respect to $z$, and putting $z=0$, we obtain
\begin{eqnarray*}
\mathscr{W}_{x}(x,y)\varpi(x,y)+\mathscr{W}_{y}(x,y)\varrho(x,y)=0.
\end{eqnarray*}
Further, we know that $x\mathscr{W}_{x}(x,y)+y\mathscr{W}_{y}(x,y)=N\mathscr{W}(x,y)$. This gives the differential equation
\begin{eqnarray}
N\mathscr{W}(x,y)\varrho(x,y)+\mathscr{W}_{x}(x,y)[y\varpi(x,y)-x\varrho(x,y)]=0.
\label{orbits}
\end{eqnarray}
Its solution is rational, and this way we find the function $\mathscr{W}$. This was our method to fill the second column of the Table \ref{table1}.

\subsection{Pseudo-flows of level $0$}\label{sub5.2}
Consider the vector field
\begin{eqnarray*}
\varpi(x,y)=Ux^{2}+Vxy+Wy^{2},\quad\varrho(x,y)=-y^2.
\end{eqnarray*}
As we have seen, the case $(V+1)^{2}-4UW=0$ is exceptional, since then the manifold $\mathcal{H}_{0}$ is no longer a one-sheeted hyperboloid but rather a cone. In this case the flow is not rational, but its expression involves only rational functions and logarithms of rational functions. Indeed, suppose
\begin{eqnarray*}
Ux^2+Vxy+Wy^2=\pm(ax+by)^2-xy,\quad a,b\in\mathbb{R}.
\end{eqnarray*}
After a linear conjugation (\ref{tiesine}), where $T:(x,y)\mapsto (px+qy,y)$ is suitably chosen, we obtain the vector field $\varpi(x,y)=-x^2-xy$, $\varrho(x,y)=-y^2$. Solving in this case (\ref{parteq}) gives the solution
\begin{eqnarray*}
\phi_{\,{\rm log}}(\m{x})=\frac{xy}{(y+1)(y+x\log(y+1))}\bl\frac{y}{y+1}.
\end{eqnarray*}
This is not defined for $y\leq -1$. Nevertheless, this function does define a flow in the half plane $\mathbb{R}\times(-1,\infty)$ with a given vector field. All flows of the form $\ell^{-1}\circ\phi_{\,{\rm log}}\circ\ell$ may be called \emph{pseudo-flows of level} $0$. The orbits of the flow $\phi_{\,{\rm log}}$ are transcendental curves $\exp(yx^{-1})y=c$. As said in the Note \ref{note1}, the verificartion that $\phi_{\,{\rm log}}$ is a flow (knowing that the boundary conditions are satisfied) is tantamount to the addition formula $\log(xy)=\log x+\log y$.
\subsection{Symmetric flows}\label{subs}
Let $i$ be a $1-$BIR involution. We say that a flow $\phi(\m{x})$ is $i-$symmetric, if $\phi$ is invariant under the conjugation with $i$:
$i\circ\phi\circ i=\phi$. There are four types of involutions (see Appendix \ref{app}). We will find all symmetric flows in the most interesting case of the linear involution $i_{0}:(x,y)\mapsto (y,x)$. Other involutions, such us linear involution $i(x,y)=(-y,-x)$, can be treated similarly. These investigations might provide a cornerstone in describing the structure of the space ${\tt FL}_{N}$, and this will be treated in \cite{alkauskas2}. Flows, invariant under $i_{0}$, are given by $\phi(\m{x})=u(x,y)\bl u(y,x)$. The examples are
$\phi_{{\rm sph},\infty}(x,y)$, $\phi_{{\rm sph},1}(x,y)$ and $\phi_{{\rm tor},1}(x,y)$, as presented, respectively, by (\ref{sphi}), (\ref{sph}) and (\ref{tor}). Further, if $J(x,y)=J(y,x)$ is a $1-$homogenic function, then the level $0$ flow (\ref{lel0}) is invariant under $i_{0}$. Assume $N\geq 1$. Let $\varpi(x,y)=Ux^2+Vxy+Wy^2=Q(x,y)-xy$, $\varrho(x,y)=-y^2$, $(V+1)^2-4UW=N^2$. The identities (\ref{vecconj}) give
\begin{eqnarray*}
A(x,y)(Q(x,y)-xy)+yA_{y}(x,y)Q(x,y)&=&\varpi'(x,y),\\
-A(x,y)y^2-yA_{x}(x,y)Q(x,y)&=&\varrho'(x,y).
\end{eqnarray*}
Any vector field of a level $N$ flow can be obtained this way; we solve these equations for $A$ and $Q$ with the assumption that $\varpi'(x,y)=\varrho'(y,x)$. This gives
\begin{eqnarray}
A(x,y)(Q(x,y)-xy)+yA_{y}(x,y)Q(x,y)=-A(y,x)x^2-xA_{x}(y,x)Q(y,x).
\label{sime}
\end{eqnarray}
The function $A$ is $0-$homogenic. Let therefore
\begin{eqnarray*}
A(x,y)=A\Big{(}\frac{x}{y}\,,1\Big{)}=a(t),\quad t=\frac{x}{y},\\
Q(x,y)=y^{2}Q\Big{(}\frac{x}{y},1\Big{)}=y^2q(t).
\end{eqnarray*}
Rewriting (\ref{sime}) in terms of $q$ and $a$, we obtain
\begin{eqnarray*}
a(t)(q(t)-t)-ta'(t)q(t)=-a\Big{(}\frac{1}{t}\Big{)}t^2-a'\Big{(}\frac{1}{t}\Big{)}q\Big{(}\frac{1}{t}\Big{)}t^2;
\end{eqnarray*}
or, in a more convenient form,
\begin{eqnarray}
a(t)(q(t)-t)+a\Big{(}\frac{1}{t}\Big{)}t^2=ta'(t)q(t)-a'\Big{(}\frac{1}{t}\Big{)}q\Big{(}\frac{1}{t}\Big{)}t^2.
\label{form}
\end{eqnarray}
(To avoid confusion: here $a'$ stands for the derivative, as opposed to $\varpi'$ and $\varrho'$, which simply stands for ``prime").
The right hand side $R(t)$ is invariant under the transformation
$R(t)\mapsto -t^3R(1/t)$, so must be the left hand side. This gives
\begin{eqnarray}
a\Big{(}\frac{1}{t}\Big{)}=-\frac{a(t)q(t)}{t^3q(\frac{1}{t})}.
\label{aa}
\end{eqnarray}
Now we use this and its derivative, and plug the results into (\ref{form}). The derivative $a'(t)$ cancels out, and this gives the equation for $a(t)$ and $q(t)$:
\begin{eqnarray*}
a(t)\Big{[}-t^2q\Big{(}\frac{1}{t}\Big{)}-q(t)+
q'(t)t^{2}q\Big{(}\frac{1}{t}\Big{)}-2q(t)tq\Big{(}\frac{1}{t}\Big{)}+q(t)q'\Big{(}\frac{1}{t}\Big{)}\Big{]}=0.\\
\end{eqnarray*}
Since $a(t)$ is not a $0-$function, the second multiplier must vanish identically. Rewrite this in terms of $q(t)=Ut^2+(V+1)t+W$ and compare the coefficients at the corresponding powers of $t$. We obtain
\begin{eqnarray}
\left\{\begin{array}{c@{\qquad}l}
t^2: W+U+(W-U)(V+1)=0,\\
t^1: (V+1)+W^2-U^2=0.
\end{array}\right.
\label{koef}
\end{eqnarray}
Power $t^{0}$ produce the same equality as $t^2$, and powers $t^3$ and
$t^{-1}$ produce identities with no conditions on $U,V,W$ (this is not surprising, since the expression $K(t)$ in the square brackets above is invariant under the transformation $K(t)\mapsto K(1/t)t^2$). Solving the system (\ref{koef}) for $U,V,W$, we obtain two cases.\\
\noindent\fbox{\textbf{A.}} $(U,V,W)=(U,-1,-U)$, $U\in\mathbb{R}$. The level of the flow is obtained from $(V+1)^2-4UW=4U^2=N^2$. So we get
\begin{eqnarray*}
q(t)=\frac{N\varepsilon}{2}\,t^2-\frac{N\varepsilon}{2},\quad\varepsilon=\pm 1.
\end{eqnarray*}
Plugging this into (\ref{aa}) gives
\begin{eqnarray*}
a\Big{(}\frac{1}{t}\Big{)}=\frac{a(t)}{t}.
\end{eqnarray*}
The general solution of this equation is given by
\begin{eqnarray*}
a(t)=(t+1)g(t),
\end{eqnarray*}
where $g$ is an arbitrary reciprocal rational function: $g(t)=g(1/t)$. So, returning back to $A$ and $Q$, we obtain
\begin{eqnarray*}
A(x,y)=\frac{x+y}{y}\cdot G(x,y),\quad Q_{N,\varepsilon}(x,y)=\frac{N\varepsilon}{2}x^2-\frac{N\varepsilon}{2}y^2,
\end{eqnarray*}
where $G$ is an arbitrary $0-$homogenic, symmetric function: $G(x,y)=G(y,x)$. As a direct check shows, $A$ and $Q_{N,\varepsilon}$ in such form indeed satisfy (\ref{sime}). Now we will give some examples of symmetric flows in this form. The vector field $Q_{N,\varepsilon}(x,y)-xy\bl (-y^2)$ arises from the univariate flow $\mathcal{W}^{(N)}_{\sigma,\tau}$ (see (\ref{uniN})), where
\begin{eqnarray*}
\sigma=-\frac{N\varepsilon}{2},\quad\tau=-\varepsilon,
\end{eqnarray*}
as (\ref{vecc}) shows. When $\varepsilon=-1$ or $1$, this gives the flows in the univariate form, respectively,
\begin{eqnarray*}
\psi(\m{x})&=&\frac{(y+1)^{N}(x+y)+(x-y)}{(y+1)^{N}(x+y)-(x-y)}\cdot\frac{y}{y+1}\bl\frac{y}{y+1},\text{ and}\\
\psi'(\m{x})&=&\frac{(y+1)^{N}(x-y)+(x+y)}{-(y+1)^{N}(x-y)+(x+y)}\cdot\frac{y}{y+1}\bl\frac{y}{y+1}.
\end{eqnarray*}
Let $A(x,y)=\frac{x+y}{y}$. Conjugating with $\ell$ of the form (\ref{birat}) we obtain two basic symmetric level $N$ flows:
\begin{eqnarray*}
\Phi_{N}(x,y)=\frac{(x+y+1)^{N}(x+y)+x-y}{2(x+y+1)^{N+1}}&\bl&
\frac{(x+y+1)^{N}(x+y)+y-x}{2(x+y+1)^{N+1}},\\
\Phi_{N}'(x,y)=\frac{(x+y+1)^{N}(x-y)+x+y}{2(x+y+1)}&\bl&
\frac{(x+y+1)^{N}(y-x)+x+y}{2(x+y+1)}.
\end{eqnarray*}
In fact, these satisfy $\Phi_{N}'(x,y)=\Phi_{-N}(x,y)$. As we will shortly see, for $N\geq 2$ this gives the complete answer to our inquiry:
\begin{prop}
Let $N\geq 2$. Then the only $i_{0}-$symmetric level $N$ flows are given by
\begin{eqnarray*}
\phi(\m{x})=\ell^{-1}\circ\Phi_{\pm N}\circ\ell,
\end{eqnarray*}
where $\ell$ is given by (\ref{birat}) for any symmetric, $0-$homogenic $A$. \label{prop13}
\end{prop}
\noindent\fbox{\textbf{B.}} Apart from a case described in \textbf{A}, the system (\ref{koef}) has another solution
\begin{eqnarray*}
W-U=\varepsilon=\pm 1, \quad (V+1)=-\varepsilon(U+W).
\end{eqnarray*}
Thus,
\begin{eqnarray*}
(U,V,W)=(U,-2-2U\varepsilon,U+\varepsilon),\quad U\in\mathbb{R}.
\end{eqnarray*}
In this case we obtain
\begin{eqnarray*}
N^{2}=(V+1)^{2}-4UW=(1+2U\varepsilon)^2-4U(U+\varepsilon)=1.
\end{eqnarray*}
As a consequence, this proves Proposition \ref{prop13}, since it shows that symmetric flows of level $N\geq 2$ arise only from the case \textbf{A}, and we are left to investigate symmetric level $1$ flows. First, we have
\begin{eqnarray*}
q(t)=Ut^{2}-(1+2U\varepsilon)t+(U+\varepsilon).
\end{eqnarray*}
As we know, the representation of a level $1$ flow in an univariate form is not unique; we have seen in Proposition \ref{mthm2} that in the simplest form, the first coordinate of a vector field is either $-\tau x^{2}$, $\tau\in\mathbb{R}$, or $-2xy$. Consider the first case. This gives
\begin{eqnarray*}
q(t)-t=Ut^{2}-(2+2U\varepsilon)t+(U+\varepsilon)=-\tau t^2.
\end{eqnarray*}
Thus, $q(t)=t^2+t$ or $q(t)=-t^2+t$.
In case vector field of a univariate flow is $-2xy$, we would get
\begin{eqnarray*}
Ut^{2}-(2+2U\varepsilon)t+U+\varepsilon=-2t,
\end{eqnarray*}
and this does not have a solution. Assume, $q(t)=t^2+t$ or $q(t)=-t^2+t$. Plugging this into (\ref{aa}) produces, respectively,
\begin{eqnarray*}
a\Big{(}\frac{1}{t}\Big{)}=-a(t),\text{ or }a\Big{(}\frac{1}{t}\Big{)}=a(t).
\end{eqnarray*}
The general solutions are
\begin{eqnarray*}
a(t)=\frac{t+1}{t-1}g(t)\text{ and }a(t)=g(t),
\end{eqnarray*}
where $g$ is an arbitrary reciprocal rational function: $g(t)=g(1/t)$. Returning back to $A$ and $Q$, we obtain
\begin{eqnarray*}
A(x,y)=\frac{x+y}{x-y}\,G(x,y),&\quad & Q(x,y)=x^2+xy,\text{ and, respectively,}\\
A(x,y)=G(x,y),&\quad & Q(x,y)=-x^2+xy,\\
\end{eqnarray*}
where, as before, $G(x,y)$ is arbitrary $0-$homogenic, symmetric function. Consider the first case. Let $A(x,y)=\frac{x+y}{x-y}$, $\ell$ be given by (\ref{birat}), and let us define the flow $\Psi(\m{x})$ by
\begin{eqnarray*}
\Psi(\m{x})=\ell^{-1}\circ\Big{(}\frac{x}{1-x}\bl\frac{y}{y+1}\Big{)}\circ\ell.
\end{eqnarray*}
Finally, recall that $\Phi_{N}$ and $\Phi'_{N}$ are symmetric flows in case $N=1$, too. This gives the following
\begin{prop}
There are $4$ basic $i_{0}-$symmetric level $1$ flows:
\begin{eqnarray*}
\phi_{\rm{tor},1}(\m{x})&=&\frac{x}{x+1}\bl\frac{y}{y+1},\\
\Phi_{1}(\m{x})&=&\frac{(x+y)^{2}+2x}{2(x+y+1)^2}\bl
\frac{(x+y)^2+2y}{2(x+y+1)^2},\\
\Phi_{1}'(\m{x})&=&\frac{x^2-y^2+2x}{2(x+y+1)}\bl \frac{y^2-x^2+2y}{2(x+y+1)},\\
\Psi(\m{x})&=&\frac{2x^2y(x+y)+(x-y)^2x}{(y-x)(x^2+xy-x+y)}\bl
\frac{2xy^2(x+y)+(x-y)^2y}{(x-y)(y^2+xy-y+x)}.
\end{eqnarray*}
All other symmetric level $1$ flows are obtained from these via conjugation (\ref{birat}), where $A(x,y)$ $0-$homogenic and symmetric.
\end{prop}

\subsection{Duality map in the ${\tt FL}_{N}$}\label{duality} Let $N\geq 2$. The hyperboloid $\mathcal{H}_{N}$ possesses many symmetries. As an example, we consider the symmetry with respect to its center $(0,-1,0)$. Let $\phi\in{\tt FL}_{N}$. There exist the unique $\ell$ of the form (\ref{birat}) such that $\ell^{-1}\circ\phi\circ\ell$ is in a univariate form $\mathcal{U}(x,y)\bl\frac{y}{y+1}$. Let the vector field of $\mathcal{U}(x,y)$ be $Ax^2+Bxy+Cy^2$. We know that $(A,B,C)\in\mathcal{H}_{N}$. Let $(A,B,C)^{\star}=(-A,-B-2,-C)$ be a reflection of it with respect to the center of the hyperboloid. Let $\mathcal{U}^{\star}(x,y)$ be an univariate flow which corresponds to $(A,B,C)^{\star}$. Let, $i_{0}(x,y)=(y,x)$ be, as before, the linear involution. We define
\begin{eqnarray*}
d:{\tt FL}_{N}\mapsto{\tt FL}_{N},\quad d(\phi)=\phi^{\star},\text{ where }
\phi^{\star}(x,y)=i_{0}\circ\ell\circ\mathcal{U}^{\star}(x,y)\circ\ell^{-1}\circ i_{0}(x,y).
\end{eqnarray*}
We call $\phi^{\star}(\m{x})$ \emph{the dual} of the flow $\phi(\m{x})$. The map $d$ is then the involution of the space ${\tt FL}_{N}$. For example, the dual of the canonical flow $\phi_{N}$ is the flow
\begin{eqnarray*}
\frac{x}{x+1}\bl y(x+1)^{-N-1}.
\end{eqnarray*}
Further, if we look closely how the flows $\Phi_{N}(\m{x})$ and $\Phi'_{N}(\m{x})$ are obtained (see the Subsection \ref{subs}), it is enough to notice that vector fields of flows $\psi$ and $\psi'$ correspond to the points mutually symmetric with respect to the center of the hyperboloid, so $\Phi_{N}(\m{x})$ and $\Phi'_{N}(\m{x})$ are in fact dual flows. The duality is indeed a nice property, and this will also be treated in \cite{alkauskas2}.

\subsection{Surfaces, invariant measures, dynamics}\label{sub5.6}It seems that any (if not every) flow can be naturally defined as a flow on a compact or a non-compact surface, if we suitably compactify the plane. For example, the flows (\ref{tor}) are flows on the torus, the flows (\ref{sphi}) and (\ref{sph}) are flows on the sphere, the flow (\ref{proj}) is a flow on the projective plane, and so on. Further, the orbits of canonical flows $\phi_{N}$ behave differently depending on the sign and parity of $N$. The surface can be defined as a quotient space, after we have examined the characteristics of a vector field; there are many standard results available which allow us to calculate the index of a vector field at a singular point \cite{krasno}. These questions will be treated in detail in \cite{alkauskas2}.
Also, rational flows will be explored in the setting of dynamical systems, as in \cite{nikolaev}. For example, the important question is to determine the invariant measure, entropy of a rational flow. 
\subsection{Higher dimensional and other investigations}The next step in describing rational flows is to go to higher dimensions. The dimension $3$ seems to be tractable, at least over complex numbers, since we know the exact structure of the Cremona group $Cr(P^{2}(\mathbb{C}))$. We have seen that for each non-negative integer $N$, $2-$dimensional projective flows of level $N$ can be continuously transformed into one another. Indeed, if $\phi_{1}$ and $\phi_{2}$ are two level $N$ flows, and $\ell_{1}$ and $\ell_{2}$ are the two $1$-BIR whose existence is guaranteed by the Theorem, then we only have to consider the flows given by conjugating of $\phi_{N}$ with the family of $1-$BIR given by $t\ell_{1}+(1-t)\ell_{2}$, $t\in[0,1]$. We may ask, for example, whether it is possible to similarly classify all $3-$dimensional flows of the same form $\frac{1}{z}\phi(\m{x}z)$, $\m{x}\in\mathbb{C}^{3}$. For example, consider the flows
\begin{eqnarray*}
\psi_{N,M}(x,y,w)&=&x(w+1)^{N-1}\bl y(w+1)^{M-1}\bl \frac{w}{w+1},\quad N,M\in\mathbb{Z}.
\end{eqnarray*}
The standard quadratic transformation of $P^2(\mathbb{C})$ is given by $(x:y:w)\mapsto (yw:xy:xw)$; the 1-BIR, associated with it, is defined by
\begin{eqnarray*}
\ell(x,y,w)=\frac{yw}{x}\bl y\bl w.
\end{eqnarray*}
This is an involution, and we can calculate:
\begin{eqnarray*}
\ell\circ\psi_{N,M}\circ\ell=x(w+1)^{M-N-1}\bl y(w+1)^{M-1}\bl\frac{w}{w+1}.
\end{eqnarray*}
Further, if we consider the linear involution $l(x,y,w)=y\bl x\bl w$, the conjugation with it gives
\begin{eqnarray*}
l\circ\psi_{N,M}\circ l=x(w+1)^{M-1}\bl y(w+1)^{N-1}\bl\frac{w}{w+1}.
\end{eqnarray*}
Finally, the conjugation with an involution $i(x,y,w)=\frac{w^2}{x}\bl y\bl w$ gives
\begin{eqnarray*}
i\circ\psi_{N,M}\circ i=x(y+1)^{-N-1}\bl y(w+1)^{M-1}\bl\frac{w}{w+1}.
\end{eqnarray*}
These equalities give the following corollary: the flows with the following indices
\begin{eqnarray*}
(N,M),\quad (-N,M),\quad (M-N,M),\quad (M,N)
\end{eqnarray*}
are 1-BIR conjugate. Therefore, among the flows $\psi_{N,M}$, one can choose the canonical family
\begin{eqnarray*}
\phi_{N}(x,y,w)&=&x(w+1)^{N-1}\bl \frac{y}{w+1}\bl \frac{w}{w+1},\quad N\in\mathbb{N}_{0}.
\end{eqnarray*}
Any other flow $\psi_{N,M}$ is 1-BIR conjugate to $\phi_{d}$, where $d=\gcd(N,M)$; here $\gcd(N,0)=N$ and $\gcd(0,0)=0$. We may ask whether the analogue of our main Theorem holds in dimension $3$ for the family $\phi_{N}$. It is premature to anticipate the answer. We may also wonder what is the precise formulation in case we choose flows over $\mathbb{R}$ rather than $\mathbb{C}$. These and similar questions concerning higher dimensional projective flows can prove to be very interesting and involved.\\

The arithmetic nature of flows even in dimension $2$ can be of big interest, too. For example, we may ask for the exact classification of all flows defined over the field $\mathbb{Q}$, or any other number field. It might happen that two flows are defined over $\mathbb{Q}$ (all their coefficients are rational numbers), and they are $\ell-$conjugate over $\mathbb{R}$ but not over $\mathbb{Q}$. If this arithmetic direction proves to be structurally interesting, we hope to return to this point in the future.
\appendix\section{Summary of facts from birational geometry}\label{app}
In this appendix we collect some facts from a birational geometry which are needed for our purposes. All these results are classical; nevertheless,we sometimes include a short indication how to prove them.

\subsection{Birational 1-homogenic maps} A map $u(\m{x})=v(x,y)\bl w(x,y)$ is called a \emph{birational plane map}, if both $u$ and $v$ are rational function in $(x,y)$, and there exists another rational map $u'(\m{x})$ such that $u'\circ u(x,y)=u\circ u'(x,y)=x\bl y$. Such $u$ is called \emph{$1-$homogenic}, if both $v$ and $w$ are homogenic rational functions of degree $1$. A group of birational automorphisms of $n-$dimensional projective space $P^{n}(\mathbb{R})$ is called the \emph{Cremona group}, denoted by $Cr(P^{n}(\mathbb{R}))$. One can explicitly describe the structure of $Cr(P^{n}(\mathbb{R}))$ when $n=1$ or $n=2$. Since in our problem the functions $v$ and $w$ are homogenic, we need only the structure of this group in dimension $1$, which is the group of linear-fractional, or M\"{o}bius, transformations. \\

Let $P(x,y)$ and $Q(x,y)$ be two homogenic polynomials of degree $d$, $d\geq 0$. Let us define, as already introduced,
\begin{eqnarray*}
\quad\ell_{P,Q}(x,y)=\frac{xP(x,y)}{Q(x,y)}\bl\frac{yP(x,y)}{Q(x,y)}.
\end{eqnarray*}
In what follows, $L$ will denote a certain non-degenerate linear transformation, $L:(x,y)\mapsto(ax+by,cx+dy)$. We will need some basic properties of the functions $\ell_{P,Q}$.
\begin{prop}We have
\begin{eqnarray*}
\ell_{P,P}(x,y)&=&{\rm id},\\
L\circ\ell_{P,Q}(x,y)&=&\ell_{P\circ L^{-1},Q\circ L ^{-1}}\circ L(x,y),\\
\ell_{P,Q}\circ\ell_{P',Q'}(x,y)&=&\ell_{P\cdot P',Q\cdot Q'}(x,y),\\
\ell_{P,Q}\circ\ell_{Q,P}(x,y)&=&{\rm id}.
\end{eqnarray*}
Here $``\cdot"$ stands for the product of two polynomials.
\label{bas}
\end{prop}
\begin{proof}
The first three properties are established by a direct check. The last one is a consequence of the first and the third one.
\end{proof}
We see that $\ell^{-1}_{P,Q}=\ell_{Q,P}$ and thus these functions are birational $1-$homogenic transformations of the plane.
\begin{prop}All birational $1$-homogenic transformations of the plane are given by
\begin{eqnarray*}
\ell_{P,Q}\circ L(x,y)=(P,Q;L),
\end{eqnarray*}
where $L$ is an invertible linear transformation, and $P(x,y)$ and $Q(x,y)$ are $n$th degree homogenic polynomials. These transformations form a group, denoted by $Bir_{{\rm h}}(\mathbb{R}^{2})$. The composition law $``\otimes"$ is given by
\begin{eqnarray}
(P_{1},Q_{1};L_{1})\otimes(P_{2},Q_{2};L_{2})&=&(P_{1}\cdot P_{2}\circ L_{1}^{-1},Q_{1}\cdot Q_{2}\circ L_{1}^{-1};L_{1}\circ L_{2}),\label{comp}\\
(P,Q;L)^{-1}&=&(Q\circ L,P\circ L;L^{-1}).\nonumber
\end{eqnarray}
Finally, $(P_{1},Q_{1};L_{1})=(P_{2},Q_{2};L_{2})$ only if $L_{1}=L_{2}$ and $\frac{P_{1}}{Q_{2}}=\frac{P_{2}}{Q_{2}}$, or $L_{1}=-L_{2}$ and $\frac{P_{1}}{Q_{2}}=-\frac{P_{2}}{Q_{2}}$.
\label{prop16}
\end{prop}
\begin{proof} (Sketch). This is a standard fact. If $\ell(x,y)=p(x,y)\bl q(x,y)$ is a birational transformation such that both $p$ and $q$ are rational $1-$homogenic functions, then $(x:y)\mapsto (p:q)$ is a birational transformation of $P^{1}(\mathbb{R})$.Thus,
\begin{eqnarray*}
\frac{p(x,y)}{q(x,y)}=\frac{ax+by}{cx+dy}, ad-bc\neq 0,
\end{eqnarray*}
and the conclusion follows. All other facts follow directly from Proposition \ref{bas}. \end{proof}

\subsection{Involutions}A second order element in $Bir_{{\rm h}}(\mathbb{R}^{2})$ is called an \emph{involution}. We will soon see that there are four classes of involutions.
\begin{prop}All involutions are given by
\begin{eqnarray*}
i^{+}_{+}(P;L)&=&(P,P\circ L;L),\quad i^{+}_{-}(P;L)=(P,-P\circ L;L);\text{ here }L^{2}={\rm id};
\end{eqnarray*}
or $\deg{P}=n$ is an odd positive integer, and
\begin{eqnarray*}
i^{-}_{+}(P;L)&=&(P,P\circ L;L),\quad i^{-}_{-}(P;L)=(P,-P\circ L;L);\text{ here }L^{2}={\rm -id}=(-x)\bl(-y).
 \end{eqnarray*}
\end{prop}
\begin{proof}Using (\ref{comp}), one checks directly that $i^{\star}_{\star}(P;L)$ are indeed all involutions. Now assume that $(P,Q;L)^{2}={\rm id.}$, $P(x,1)$ and $Q(x,1)$ are co-prime polynomials. Then
\begin{eqnarray*}
(P\cdot P\circ L^{-1},Q\cdot Q\circ L^{-1};L^{2})={\rm id}.
\end{eqnarray*}
Thus, $L^{2}={\rm id}$ and $P\cdot P\circ L^{-1}=Q\cdot Q\circ L^{-1}$; or $L^{2}={\rm -id}$ and $P\cdot P\circ L^{-1}=-Q\cdot Q\circ L^{-1}$. Since $P$ and $Q$ have no common factors, the first case implies $P=cQ\circ L^{-1}$ for a certain constant $c\neq 0$. Thus, $P\circ L^{-1}=cQ\circ L^{-2}=
cQ$, and this gives $c^2=1$. Now, assume the second case. Then $P=cQ\circ L^{-1}$ for $c\neq 0$. This implies $P\circ L^{-1}(x,y)=cQ\circ L^{-2}(x,y)=cQ(-x,-y)=(-1)^{n}cQ(x,y)$. Thus, $(-1)^{n}c^{2}=-1$. This gives $c=\pm 1$, and $n$ is an odd positive integer. \end{proof}
We can confine to two types of involutions $i_{+}:=i^{+}_{+}$ and $i_{-}:=i^{-}_{+}$, since the other two, when applied to a rational flow, only change the direction, as compared to the action of $i^{+}_{+}$ and $i^{-}_{+}$. More precisely, if $\phi$ is a flow, and $L^{2}={\rm id.}$, then
\begin{eqnarray*}
(P,P\circ L;L)\circ\phi\circ(P,P\circ L;L)=(P,-P\circ L;L)\circ\phi^{-1}\circ(P,-P\circ L;L).
\end{eqnarray*}
This follows from the following calculations derived directly from (\ref{comp}):
\begin{eqnarray*}
(P,P\circ L;L)\circ(P,-P\circ L;L)(x,y)=(-x,-y).
\end{eqnarray*}
The example of involution of the type $i_{+}$ is $\frac{y^2}{x}\bl y$.
Equally, if $\phi$ is a flow, and $L^{2}=-{\rm id}$, then we also have
\begin{eqnarray*}
(P,P\circ L;L)\circ(P,-P\circ L;L)(x,y)=(-x,-y).
\end{eqnarray*}
The example of involution of the type $i_{-}$ is
\begin{eqnarray*}
-\frac{(x-y)^2}{x}\bl\frac{(x-y)(y-2x)}{x}.
\end{eqnarray*}

\end{document}